\documentclass[11pt,reqno]{amsart}
\usepackage{amssymb}
\usepackage{amscd,enumerate,amsfonts,calc,amsmath,verbatim,xypic}

\usepackage{amsmath, amssymb, amsfonts, amstext, amsthm, amscd, xypic}
\usepackage[mathscr]{eucal}
\usepackage{enumerate}
\usepackage{tikz,color,soul}
\usetikzlibrary{matrix,arrows}
\usepackage[utf8]{inputenc}
\usepackage[english]{babel}

\usepackage{mathrsfs}

\usepackage{verbatim}

\usepackage[new]{old-arrows}

\usepackage{amsmath,calligra,mathrsfs}
\DeclareMathOperator{\Hom}{\mathscr{H}\text{\kern -3pt {\calligra\large om}}\,} 

\usepackage{longtable}
\usepackage{hyperref}
\usepackage[all]{xy}
\usepackage{ulem}
\usepackage{xcolor}

\usetikzlibrary{arrows.meta, positioning,arrows}

\newtheorem{thm}{Theorem}[section]
\newtheorem{prop}[thm]{Proposition}
\newtheorem{lemma}[thm]{Lemma}
\newtheorem{cor}[thm]{Corollary}
\newtheorem{defn}[thm]{Definition}

\newtheorem{rmk}[thm]{Remark}

\numberwithin{equation}{section}

\def\rar{{\rightarrow}}

\newcommand{\R}{\mathbb R}
\newcommand{\Z}{\mathbb Z}

\def \mc{\mathcal}

\allowdisplaybreaks
\begin{document}
	
	\title{Logarithmic connections on principal bundles over normal varieties}

	\author{Jyoti Dasgupta}
	\address{School of Mathematics, Tata Institute of Fundamental Research, Mumbai, India}
	\email{jdasgupta.maths@gmail.com}
	
	\author{Bivas Khan}
	\address{Department of Mathematics, Indian Institute of Science Education and Research-Pune, Pune, India}
	\email{bivaskhan10@gmail.com}
	
	\author{Mainak Poddar}
	\address{Department of Mathematics, Indian Institute of Science Education and Research-Pune, Pune, India}
	\email{mainak@iiserpune.ac.in}
	
	\subjclass[2020]{14J60, 14M25, 32L05, 53C05.}
	
	\keywords{Logarithmic connection, principal bundle, vector bundle, residue, normal variety, toric variety}

	\begin{abstract}
		Let $X$ be a normal projective variety over an algebraically closed field of characteristic zero. Let $D$ be a reduced Weil divisor on $X$. Let $G$ be a reductive linear algebraic group. We introduce the notion of a logarithmic connection on a principal $G$-bundle over $X$, which is singular along $D$. The existence of a logarithmic connection on the frame bundle associated with a vector bundle over $X$ is shown to be equivalent to the existence of a logarithmic covariant derivative on the vector bundle if the logarithmic tangent sheaf of $X$ is locally free. Additionally, when the algebraic group $G$ is semisimple, we show that a principal $G$-bundle admits a 
  logarithmic connection if and only if the associated adjoint bundle
  admits one. We also prove that the existence of a logarithmic connection on a principal bundle over a toric variety, singular along the boundary divisor, is equivalent to the existence of a torus equivariant structure on the bundle. 
		
	\end{abstract}
	
	\maketitle
	
	\tableofcontents
	
	\section{Introduction}
	A connection on a principal bundle over a manifold generalizes the notion of a directional derivative, allowing us to compare nearby fibers. Connections always exist in the differential category, although we may have some obstruction for integrablity of the connection. The study of holomorphic connection for holomorphic principal bundles was initiated by Atiyah \cite{At}. He defined the notion of a holomorphic connection on a principal bundle as a splitting of a certain short exact sequence of vector bundles, known as the Atiyah sequence. However, the existence of a connection in the holomorphic category is not guaranteed. In fact, by \cite[Theorem 4]{At}, if a holomorphic principal \(G\)-bundle on a complex K\"ahler manifold admits a holomorphic connection where \(G\) is semisimple or the general linear group, then its characteristic cohomology ring must vanish.  This puts a severe restriction on the class of holomorphic principal bundles or vector bundles, which admit a holomorphic connection. Therefore, to have the notion of connection for a larger class of interesting holomorphic principal bundles, the holomorphicity condition of a connection is usually relaxed and  the connection is allowed to have poles along certain divisors. In particular, one can ask if there exists a logarithmic connection on a principal bundle, i.e., a connection with poles of at most order one along a given divisor. The study of logarithmic connections on holomorphic vector bundles was initiated by Deligne \cite{deligne} and leads toward the Riemann-Hilbert correspondence. Logarithmic connections also have applications to vanishing theorems \cite{EV1, EV2}. Recently, the work of Biswas et al. \cite{BMS, BDP, BKN, BP} has revealed a link between logarithmic connections on holomorphic principal bundles and group actions.


	
	To what extent does the theory of logarithmic connections on principal bundles extend to singular algebraic varieties?  
	In this paper, we take a step towards addressing this question by studying logarithmic connections singular along a reduced Weil divisor on a principal bundle over a normal algebraic variety. We fix the base field $\mathbb{K}$ to be an algebraically closed field of characteristic zero. To be precise, let \(G\) be a reductive algebraic group and $X$ be a normal projective variety together with a reduced Weil divisor $D$.  We extend the notion of the Atiyah sequence to  algebraic principal \(G\)-bundles over $X$ (see Proposition \ref{atseq}).  In contrast to the holomorphic situation, this is only a short exact sequence of torsion-free sheaves.
 As a corollary of Proposition \ref{atseq}, we also obtain a logarithmic Atiyah sequence.  This enables us to define the notion of a logarithmic connection singular along $D$ in this setup. 
 
 The logarithmic Atiyah sheaf is locally free if the logarithmic tangent sheaf is so (for instance, in the case of toric varieties and log smooth pairs $(X, D)$).  In Section \ref{residue sec}, assuming that the logarithmic tangent sheaf is locally free, we show that the existence of a logarithmic connection on the frame bundle associated to a vector bundle is equivalent to the existence of a logarithmic covariant derivative on the vector bundle (see Theorem \ref{a=b}). Moreover, a logarithmic connection on the frame bundle induces a natural logarithmic connection on the vector bundle (see Proposition \ref{equivalence1}). More generally, we also study the existence of logarithmic connection on the associated bundle of a principal bundle corresponding to a homomorphism of the structure group (see Section \ref{sec 4.1}). These were studied in the holomorphic case by  Gurjar and Paul \cite{GP}. Under the additional assumption that \(G\) is semisimple, we show that the existence of a logarithmic connection on the principal bundle $\mathcal{P}$ is equivalent to the existence of a logarithmic connection on the associated adjoint bundle \(\text{ad}(\mathcal{P})\) (see Corollary \ref{adbdlstab}).


	Toric varieties serve as a good testing ground for many theories in algebraic geometry. In Section \ref{toric situation}, we focus on studying logarithmic connections on principal bundles over toric varieties. Let \(X\) be a projective  toric variety under the action of the algebraic torus \(T\). Let \(O\) be the dense open subset of \(X\) on which \(T\) acts freely. Then \(D:= X \setminus O\) is the boundary divisor.  When the toric variety is smooth and the base field is complex numbers, in \cite{BMS}, the authors proved that a holomorphic vector bundle over \(X\) admits a logarithmic connection singular over \(D\) if and only if the vector bundle admits a $T$-equivariant structure. For a connected reductive group \(G\), this result was generalized to the case of holomorphic principal \(G\)-bundles over \(X\) in \cite{BDP}. In Subsections \ref{log conn toric} and \ref{eqslc}, we extend this result to singular toric varieties defined over $\mathbb{K}$: A  principal \(G\)-bundle \(\mathcal{P}\) over a projective toric variety \(X\)  admits a logarithmic connection singular over the boundary divisor \(D\) if and only if the bundle is $T$-equivariant (see Theorem \ref{logconn2equi}). 
	
	    In the case of toric varieties, the  logarithmic tangent sheaf is  free. 
		Using this, we explicitly describe the residues of a logarithmic connection on a vector bundle over a normal toric variety in Subsection \ref{ssc: residue}. Consider the logarithmic connection induced from the torus action on an equivariant vector bundle over a normal toric variety \(X\). Following \cite{BMS}, we explain how its logarithmic residues have information on the equivariant structure. We also point out a relation of these residues with the equivariant Chern class \cite{payne} of the vector bundle.


 There is a generalization of logarithmic differentials for an arbitrary scheme, given by the notion of \textit{log structures} (see \cite{Illusie}, \cite{K.Kato}). In fact, there is a notion of logarithmic connection and its residue on vector bundles over fine log schemes (see \cite{Mau}). When \(X\) is a smooth projective variety and \(D\) is a normal crossing divisor on \(X\), then we have a natural log smooth, in particular, a fine log structure on \(X\). However, if the underlying variety \(X\) is a normal projective variety and we consider a reduced Weil divisor \(D\) on \(X\), then we may not have a fine log structure. 
	
	\subsection*{Acknowledgments} The authors thank Indranil Biswas and Vikraman Balaji for stimulating and helpful conversations. They also thank Henri Guenancia for explaining a technical point to them.  Special thanks go to Barabara Fantechi for her interest in the work and for detailed discussions that led to significant improvements in the manuscript.  The first and second named authors are partly supported by postdoctoral fellowships from the National Board for Higher Mathematics, Department of Atomic Energy, India. The second author is also partially supported by EPSRC grant EP/R02300X/1. The research of the third author was supported by a SERB MATRICS Grant: MTR/2019/001613.


	\section{Preliminaries}

	In this section, we briefly review some basic facts about logarithmic tangent sheaves and toric varieties. Henceforth, we denote the base field by $\mathbb{K}$, assumed to be algebraically closed of characteristic zero. Unless otherwise mentioned, our varieties are normal projective varieties over $\mathbb{K}$.

	\subsection{Logarithmic tangent sheaves}
	
	Let \(X\) be a smooth projective algebraic variety and \(D\) be a simple normal crossing divisor. Then the pair \((X, D)\) gives rise to a natural sheaf \(\Omega_X^1(\text{log }D)\) of differential 1-forms on \(X\) with logarithmic poles on \(D\). Originally, this sheaf was introduced by Deligne \cite{del2} to define a mixed Hodge structure on the open variety \(X \setminus D\). In this case, the logarithmic cotangent sheaf is locally free and so is its dual, the logarithmic tangent sheaf. Later, Saito extended these notions for any divisor on a smooth variety and discussed many useful properties in his fundamental paper \cite{saito}. Now, more generally, let us consider a normal projective variety \(X\) and a reduced Weil divisor \(D\)  on \(X\). Define the logarithmic tangent sheaf of \((X, D)\), denoted as \(\mathscr{T}_X(- \text{log } D)\) to be the sheafification of the module of derivations preserving \(D\) (see \cite[Section 3.1]{gue} for more details). This is a subsheaf of the tangent sheaf \(\mathscr{T}_X\), defined as the sheafification of the module of derivations. We no longer have the local freeness of the logarithmic tangent sheaf, although it is still a reflexive sheaf since the underlying variety is normal. Its dual is called the logarithmic cotangent sheaf, denoted by \(\Omega_X^1(\text{log }D):= \Hom_{\mathcal{O}_X}(\mathscr{T}_X(- \text{log } D), \mathcal{O}_X) \).

	\subsection{Toric varieties}
	We recall some basic facts about logarithmic (co)tangent sheaf on toric varieties. We assume toric varieties are normal following standard conventions (see \cite{Cox} and \cite{Oda}).	Let $T \cong \left(\mathbb{K}^*\right) ^n$ be an algebraic torus. Let $M=\text{Hom}(T, \mathbb{K}^*) \cong \Z^n$ be its character lattice, and let $N=\text{Hom}_{\Z}(M, \Z)$ be the dual lattice of one parameter subgroups. Let \(\langle \, ,  \rangle \) denote the canonical bilinear perfect pairing \(\langle \, ,  \rangle: M \times N \rightarrow \Z\). Let $\Sigma$ be a fan in $N_{\R}:=N \otimes_{\Z} \R$ which defines a toric variety $X=X(\Sigma)$ under the action of the torus $T$. Let $\Sigma(1)$ denote the rays of $\Sigma$ and $\sigma(1)$ denote the rays of a cone $\sigma$ in $\Sigma$. For a cone \(\sigma \in \Sigma\), denote the corresponding affine variety by \(U_{\sigma}\). The affine toric variety \(U_{\sigma}\) has a unique distinguished point, denoted by \(x_{\sigma}\). The orbit containing the distinguished point \(x_{\sigma}\) under the action of the torus \(T\) is denoted by \(O_{\sigma}\). Note that $O_{\sigma} = \text{Spec }\mathbb{K}[\sigma^{\perp} \cap M] $. When \(\sigma \in \Sigma(n)\), we have  \(O_{\sigma} = \{ x_{\sigma} \} \). For the trivial cone $\{0\}$, we have the principal orbit \(O:= O_{\{0\}} =U_{\{0\}}\). Moreover, there is a canonical identification of $O$ with $T$ as both are defined as $\text{Spec }\mathbb{K}[M] $ (cf. \cite[Proposition 1.6]{Oda}). We denote the closed point of $O$ corresponding to the identity element $1_T$ of $T$ by $x_0$. Each ray $\rho \in \Sigma(1)$ corresponds to a \(T\)-invariant prime divisor, denoted by \(D_{\rho}\). 
	
	Let \(X\) be a projective toric variety and \(D=X \setminus T\) be the boundary divisor. Note that \(D=\sum\limits_{\rho \in \Sigma(1)} D_{\rho}\) is an effective \(T\)-invariant divisor. Let \(\mathscr{T}_X(- \text{log } D) \subset \mathscr{T}_X\)  be the corresponding logarithmic tangent sheaf. Let \(\mathfrak{t}\) be the Lie algebra of the torus, which is also identified with the vector space of left-invariant derivations of \(\mathbb{K}[M]\). Note that we have a canonical identification 
	\begin{equation*}
		\mathbb{K} \otimes_{\Z} N \stackrel{\cong} \longrightarrow \mathfrak{t} \text{ given by } v \mapsto \delta_v,
	\end{equation*}
	where for any \(v \in N\),
	\begin{equation}\label{LID on X}
		\delta_v: \mathbb{K}[M] \rar \, \mathbb{K}[M] \text{ defined by } \chi^m  \mapsto \langle m,v \rangle \chi^m  
	\end{equation}
	is a left-invariant derivation. Consider the sheaf \(\mathcal{V}:= \mathcal{O}_X \otimes_{\mathbb{K}} \mathfrak{t}\) on \(X\). Note that, \(\text{H}^0(X, \mathcal{V}) \cong \mathfrak{t}\) and hence \(\text{H}^0(X, \mathcal{V})\) has a natural Lie algebra structure. From the torus action on \(X\), we have \(\mathfrak{t} \subset \text{H}^0(X, \mathscr{T}_X)\). Indeed, it is easy to see that \(\delta_v\), defined in \eqref{LID on X}, preserves \(\mathbb{K}[S_{\sigma}]\) for each \(\sigma \in \Sigma\) and hence defines an element of \(\text{H}^0(X, \mathscr{T}_X)\). Thus we have an \(\mathcal{O}_X\)-module homomorphism of sheaves
	\begin{equation}\label{trivial}
		\beta: \mathcal{V} \, \rar \, \mathscr{T}_X.
	\end{equation}
	
	\begin{prop}[{\cite[Proposition 3.1]{Oda}}]\label{odat}
		The homomorphism $\beta$ is an isomorphism onto \(\mathscr{T}_X(- \text{log } D) \subset  \mathscr{T}_X\), i.e. it induces an isomorphism
		\begin{equation}\label{trivial2}
			\bar{\beta}:	\mathcal{V} \stackrel{\cong}\longrightarrow \mathscr{T}_X(- \text{log } D).
		\end{equation}
	\end{prop}

	Using the above isomorphism, we can pull back the natural Lie algebra structure of \(\mathscr{T}_X(- \text{log } D)\) on the sheaf $\mathcal{V}$ and 
	regard \(
	\bar{\beta} \text{ as a Lie algebra homomorphism.}\) Explicitly, the Lie algebra structure on $\mathcal{V}$ is given in the following way. Let \(U \) be an open subset of \(X\), \(f_1, f_2 \in \mathcal{O}_X(U)\) and \(v_1, v_2 \in N\). Then 
	\begin{equation}\label{Lie_str}
		[f_1 \otimes v_1, f_2 \otimes v_2]:=f_1 \, \delta_{v_1}(f_2) \otimes v_2 - f_2 \, \delta_{v_2}(f_1) \otimes v_1.
	\end{equation} 
	Also, we have \( \Omega_X^1(\text{log }D) \cong \mathcal{O}_{X} \otimes_{\Z} M\) using Proposition \ref{odat}.

	\section{Logarithmic connections on principal bundles}\label{log conn sec}
	
	In this section, we extend the definition of logarithmic connection on a principal bundle over an arbitrary normal projective variety. For ease of notation, we denote \(\text{Der}(R, R)\) by \(\text{Der}(R)\), for any ring \(R\). 
	
	
	Let \(X\) be a normal variety and $G$ be a   linear algebraic group. Consider an algebraic principal $G$-bundle \(p : \mathcal{P} \rightarrow X\) over $X$. The \(G\)-action on \(\mathcal{P}\) induces a \(G\)-action on \(p_*\mathscr{T}_{\mathcal{P}}\) as follows. Let \(U\) be an open subset of \(X\). Then \(G\) acts from the right on \((p_*\mathscr{T}_{\mathcal{P}})(U)=\text{Der}_{\mathbb{K}}(\mc{O}_{\mathcal{P}}(p^{-1}(U)))\) by
	\begin{equation*}
		(\tilde{\delta} \cdot g)(\tilde{f} )=\tilde{\delta}(\tilde{f}  \cdot g^{-1} ) \cdot g,
	\end{equation*}
	where \(\tilde{\delta} \in \text{Der}_{\mathbb{K}}(\mc{O}_{\mathcal{P}}(p^{-1}(U))), \, \tilde{f} \in \mc{O}_{\mathcal{P}}(p^{-1}(U)), \, \text{and } g \in G\) and \(\tilde{f}  \cdot g^{-1}\) is defined by
	\begin{equation*}
		(\tilde{f}  \cdot g^{-1})(e)=\tilde{f} (e g),
	\end{equation*}
	for all \(e \in \mathcal{P}\). 
	
	\begin{defn}
		The Atiyah sheaf associated to the principal bundle is the subsheaf of \(G\)-invariants in $p_*\mathscr{T}_{\mathcal{P}}$:
		\begin{equation*}
			\begin{split}
				\mathcal{A}t(\mathcal{P}):=(p_*\mathscr{T}_{\mathcal{P}})^G \subset p_*\mathscr{T}_{\mathcal{P}}.
			\end{split}
		\end{equation*}
	\end{defn}
\begin{rmk}\rm {
Assume that \(G\) is a reductive linear algebraic group. Then, using \cite[Theorem 1.2]{Nevins}, the Atiyah sheaf $\mathcal{A}t(\mathcal{P})$ is coherent as \(p: \mathcal{P} \rightarrow X \) is a principal \(G\)-bundle. Moreover, since \(p\) is surjective, using \cite[Proposition 8.4.5]{ega1}, we get that $\mathcal{A}t(\mathcal{P})$ is torsion free.

}
	
\end{rmk}

	We now define a map \(\eta : \mathcal{A}t(\mathcal{P}) \rightarrow \mathscr{T}_X\). Let \(U\) be an open subset of \(X\) and take \(\tilde{\delta} \in \text{Der}_{\mathbb{K}}(\mc{O}_{\mathcal{P}}(p^{-1}(U)))^G\). For the projection \(p : \mathcal{P} \rightarrow X\), let \(p^{\sharp}\) denote the map of sheaves of \(\mathcal{O}_X\)-modules, 
	\begin{equation}\label{p-sharp}
		p^{\sharp}: \mathcal{O}_X \rightarrow p_*(\mathcal{O}_{\mathcal{P}}).
	\end{equation}
	Since \(p: \mathcal{P} \, \rar \, X\) is a categorical quotient, we have \(\mc{O}_{\mathcal{P}}(p^{-1}(U))^G \cong \mathcal{O}_{X}(U)\). For any \(f \in \mc{O}_X(U)\), there exists a unique element \(p_*(\tilde{\delta}(p^{\sharp}(f))) \in \mathcal{O}_{X}(U)\)  such that
	\begin{equation}\label{log_con1}
		\tilde{\delta}(p^{\sharp}(f))= p_*(\tilde{\delta}(p^{\sharp}(f))) \circ p.
	\end{equation}
	Thus, we get a map 
	\begin{equation*}
		p_*^U: \mathcal{O}_X(U) \rightarrow \mathcal{O}_X(U) \text{ given by } f \mapsto p_*(\tilde{\delta}(p^{\sharp}(f))).
	\end{equation*}
	Note that, due to uniqueness of the element \( p_*(\delta(p^{\sharp}(f)))\), the map \(p_*^U\) is additive and it satisfies the following relation
	\begin{equation}\label{p_* Leib}
		p_*^U(f_1 \, f_2)=f_1 \, p_*^U(f_2) + f_2 \, p_*^U(f_1),	
	\end{equation} 
	for \(f_1, \, f_2 \in \mathcal{O}_X(U)\). Let us define 
	\begin{equation}\label{surj}
		\begin{split}
			\eta_U: \text{Der}_{\mathbb{K}}(\mc{O}_{\mathcal{P}}(p^{-1}(U)))^G & \rightarrow \, \text{Der}_{\mathbb{K}}(\mc{O}_X(U)), 
		\end{split}
	\end{equation}
	by the following rule. For \(\tilde{\delta} \in \text{Der}_{\mathbb{K}}(\mc{O}_{\mathcal{P}}(p^{-1}(U)))^G\), and \(f \in \mc{O}_X(U)\), let \[\eta_U(\tilde{\delta})(f)=p_*^U(f).\] Since $\tilde{\delta}$ is a derivation and the map \(p_*^U\) is additive and satisfies \eqref{p_* Leib}, $\eta_U(\tilde{\delta})$ is indeed a derivation. 
	
	\begin{lemma}\label{preatseq}
		Let \(p : \mathcal{P} \rightarrow X\) be a principal \(G\)-bundle on \(X\) which is locally trivial for the Zariski topology. Then we have the following short exact sequence of sheaves
		\begin{equation}\label{preatseqeq}
			0 \longrightarrow (p_*\mathscr{T}_{\mathcal{P} / X})^G \longrightarrow \mathcal{A}t(\mathcal{P})  \stackrel{\eta} \longrightarrow  \mathscr{T}_X \longrightarrow 0.
		\end{equation}
	\end{lemma}
	
	\begin{proof}
		Let \(U \subset X\) be a trivializing open subset for the principal bundle $\mathcal{P}$, so that we have  \(\mc{O}_{\mathcal{P}}(p^{-1}(U)) \cong  \mc{O}_X(U) \otimes \, \mathcal{O}(G)\). Then for any \(\delta \in \text{Der}_{\mathbb{K}}(\mc{O}_X(U))\), let us define a derivation \(\tilde{\delta} \in \text{Der}_{\mathbb{K}}(\mc{O}_{\mathcal{P}}(p^{-1}(U)))^G\) by setting 
		\begin{equation*}
			\tilde{\delta}(f \otimes h) = \delta(f) \otimes h,
		\end{equation*}
		where \(f \in \mc{O}_X(U), \, h \in \mathcal{O}(G)\). Then by definition, \(\eta_U(\tilde{\delta})=\delta\). 
		Thus, \eqref{surj} defines a surjective map 
		\begin{equation}\label{the map eta}
			\eta: \mathcal{A}t(\mathcal{P}) \rightarrow \mathscr{T}_X.
		\end{equation}
		We show that \(\text{Ker}(\eta)=(p_*\mathscr{T}_{\mathcal{P} / X})^G\). Let \(U \) be an affine open subset of \(X\) and write \(A=\mathcal{O}_{\mc{P}}(p^{-1}(U))\) and \(B=\mathcal{O}_X(U)\). Let \(\tilde{\delta} \in \mathcal{A}t(\mathcal{P})(U)\), then we have
		\begin{equation*}
			\begin{split}
				&\tilde{\delta} \in \text{Ker}(\eta_U)  \Leftrightarrow \tilde{\delta} \text{ is \(G\)-invariant and } p_*(\tilde{\delta}(p^{\sharp}(f) )=0 \text{ for all } f \in \mathcal{O}_X(U)\\
				& \Leftrightarrow \tilde{\delta} \text{ is \(G\)-invariant and } \tilde{\delta}(p^{\sharp}(f) )=0 \text{ for all } f \in \mathcal{O}_X(U) ~ (\text{using \eqref{log_con1} and surjectivity of }p)\\
				& \Leftrightarrow \tilde{\delta} \in \text{Der}_{B}(A)^G=\text{Hom}_B(\Omega_{A/B}, A)^G=(p_*\mathscr{T}_{\mathcal{P} / X})(U)^G.
			\end{split}
		\end{equation*}	
		This completes the proof.
	\end{proof}

	\begin{rmk}\label{isotrivial}
		\rm{Let \(p : \mathcal{P} \rightarrow X\) be a principal \(G\)-bundle on \(X\). Since \(G\) is an affine algebraic group, the principal bundle $\mathcal{P}$ is \textit{locally isotrivial} by \cite[Remark 3.1]{Br}. In other words, for any point \(x \in X\) there exists a Zariski open subset \(V \subset X\) containing \(x\) and a finite étale surjective morphism \(f: U \rightarrow V\) such that the pull-back principal bundle \(f^*(\mathcal{P}|_{V})\) is trivial.}
	\end{rmk} 

In the following proposition, we show that the exact sequence \eqref{preatseqeq} exists for any principal bundle.
	
		\begin{prop}\label{atseq}
		Let \(G\) be a reductive linear algebraic group and \(p : \mathcal{P} \rightarrow X\) be a principal \(G\)-bundle on \(X\). Then we have the following short exact sequence of sheaves
	\begin{equation}\label{preatseqeq1.1}
		0 \longrightarrow (p_*\mathscr{T}_{\mathcal{P} / X})^G \longrightarrow \mathcal{A}t(\mathcal{P})  \stackrel{\eta} \longrightarrow  \mathscr{T}_X \longrightarrow 0.
	\end{equation}
\end{prop}

\begin{proof}
	Note that the exactness of a sequence can be checked by pulling back using an étale cover. In view of Remark \ref{isotrivial}, there is an étale cover \(\{f_i : U_i \rightarrow V_i\}\) of \(X\) such that the pull-back principal bundle \(f_i^*(\mathcal{P}|_{V_i})\) is trivial. Hence, it suffices to show that for an étale map \(f : U \rightarrow X \), the following sequence is exact 
		\begin{equation*}
		0 \longrightarrow f^*(p_*\mathscr{T}_{\mathcal{P} / X})^G \longrightarrow f^*\mathcal{A}t(\mathcal{P})   \longrightarrow  f^*\mathscr{T}_X \longrightarrow 0.
	\end{equation*}
	We have the pull-back diagram
	\begin{equation*}
		\begin{tikzpicture}[description/.style={fill=white,inner sep=2pt}]
			\matrix (m) [matrix of math nodes, row sep=3em,
			column sep=2.5em, text height=1.5ex, text depth=0.25ex]
			{ Q:=f^*(\mathcal{P})    & &   \mathcal{P}  \\
				U    & &   X  \\ };
			\path[->]  (m-1-1) edge node[auto] {}(m-1-3);
			\path[ ->] (m-2-1) edge node[below] {}(m-2-3);
			\path[->] (m-1-1) edge node[auto] {$q$}(m-2-1);
			\path[->] (m-1-1) edge node[above] {$f'$}(m-1-3);
			\path[ ->] (m-1-3) edge node[auto] {$p$} (m-2-3);
			\path[->] (m-2-1) edge node[above] {$f$} (m-2-3);
		\end{tikzpicture}	.
	\end{equation*}
 Observe that \(f'\) is \(G\)-equivariant and also étale. Since, \(f\) is étale, i.e., in particular, unramified and smooth, we use \cite[Ch. I, Proposition 3.5]{Milneetco} and the exactness of the relative cotangent sequence (see \cite{SP1}) to get that the natural homomorphism \(f^* \Omega_X \cong \Omega_U\) is an isomorphism. Since, in particular \(f\) is flat, taking dual we have \(\mathscr{T}_U \cong f^*\mathscr{T}_X\) (see \cite[proof of Proposition 1.8]{stabRHa}).  Now, by \cite[Ch. III, Proposition 9.3]{AGHa}, we see that the natural homomorphism \[f^{*}(p_*\mathscr{T}_{\mathcal{P} / X}) \rightarrow q_* f'^* \mathscr{T}_{\mathcal{P} / X} \] 
 is an isomorphism. Then using \cite[Ch. II, Proposition 8.10]{AGHa} together with \cite[proof of Proposition 1.8]{stabRHa}, we get an isomorphism \[f^{*}(p_*\mathscr{T}_{\mathcal{P} / X}) \cong q_*\mathscr{T}_{\mathcal{Q} / U},\] 
 which is also \(G\)-equivariant. Hence, taking \(G\)-invariants we have \(	f^{*}((p_*\mathscr{T}_{\mathcal{P} / X})^G) \cong (q_*\mathscr{T}_{\mathcal{Q} / U})^G\). Note that these sheaves are coherent using \cite[Theorem 1.2]{Nevins}, as both \(p\) and \(q\) are projection maps of principal bundles. Similarly, using the fact that both \(f\) and \(f'\) are étale, we get an isomorphism \(f^*\mathcal{A}t(\mathcal{P}) \cong \mathcal{A}t(\mathcal{Q})\). Since all these isomorphisms are induced from natural homomorphisms, we have the following commutative diagram of coherent sheaves

		\begin{equation}\label{isotrivialcommtsq}
	\xymatrixrowsep{1.8pc} \xymatrixcolsep{2.2pc}
	\xymatrix{
		0\ar@{->}[r]&
		f^{*}((p_*\mathscr{T}_{\mathcal{P} / X})^G)\ar@{->}[r]^{}\ar@{->}[d]^{\cong}&f^*\mathcal{A}t(\mathcal{P})\ar@{->}[d]^{\cong}\ar@{->}[r]^{}&f^*\mathscr{T}_X\ar@{->}[d]_{\cong}\ar@{->}[d]\ar@{->}[r]&0\\
		0\ar@{->}[r]&
		(q_*\mathscr{T}_{\mathcal{Q} / U})^G\ar@{->}[r]^{}&\mathcal{A}t(\mathcal{Q})\ar@{->}[r]^{}&\mathscr{T}_U\ar@{->}[r]&{0}.
	}
			\end{equation}
	\noindent	
    The bottom row in \eqref{isotrivialcommtsq} is exact using Lemma \ref{preatseq}. Hence, the top row is also exact. This completes the proposition.
\end{proof}

	We have the identification of 
	\(\text{Der}_{\mathbb{K}}(\mathcal{O}(G))^G\), the space of left invariant derivations of \(G\) with \(\text{Der}_{\mathbb{K}}(\mathcal{O}_{G, 1_G}, \mathbb{K})\), the tangent space of \(G\)  at the identity. We use the same notation \(\mathfrak{g}\) for both notions interchangeably.  The adjoint bundle is defined by \(\text{ad}(\mathcal{P})= \mathcal{P} \times^{G} \mathfrak{g}\), where \(G\) acts on the Lie algebra $\mathfrak{g}$ via the adjoint action.
	
	\begin{prop}\label{atseqad}
With the same assumptions as the previous proposition, we have 
		\begin{equation*}
			(p_*\mathscr{T}_{\mathcal{P} / X})^G \cong \text{ad}(\mathcal{P}).
		\end{equation*}
	\end{prop}

\begin{proof}
Since $\mathcal{P}$ is a principal \(G\)-bundle, we have a canonical isomorphism \(\mathcal{P} \times_X \mathcal{P} \cong \mathcal{P} \times G\). Thus the diagonal map \(\mathcal{P} \rightarrow \mathcal{P} \times_X \mathcal{P}\) becomes
\begin{equation*}
    j: \mathcal{P} \rightarrow \mathcal{P} \times G, \, e \mapsto (e, 1_G) \text{ for } e \in \mathcal{P}.
\end{equation*}
This gives the following exact sequence 
\[0 \rightarrow \mathscr{T}_{\mathcal{P}} \rightarrow j^* \mathscr{T}_{\mathcal{P} \times G} \rightarrow \mathcal{N}_{  \mathcal{P} / \mathcal{P} \times G } \rightarrow 0,\]
where \(\mathcal{N}_{  \mathcal{P} /    \mathcal{P} \times G }\) is the normal sheaf. Let \(p_1 : \mathcal{P} \times G \rightarrow \mathcal{P} \) and \(p_2 : \mathcal{P} \times G \rightarrow G\) denote the corresponding projections. Then, we have 
\begin{equation*}
    \begin{split}
        j^* \mathscr{T}_{\mathcal{P} \times G} \cong j^* (p_1^* \mathscr{T}_{\mathcal{P}} \oplus p_2^* \mathscr{T}_G)  \cong \mathscr{T}_{\mathcal{P}} \oplus (\mathcal{O}_{\mathcal{P}} \otimes \mathfrak{g}),
    \end{split}
\end{equation*}
where the second isomorphism follows from the fact that
\[p_1 \circ j=Id_{\mathcal{P}}, \, p_2 \circ j = 1_G \text{ and } \mathscr{T}_G \cong \mathcal{O}_G \otimes \mathfrak{g}.\]
Here the last isomorphism is \(G\)-equivariant where the \(G\)-action on the Lie algebra $\mathfrak{g}$ is given by the adjoint representation of \(G\) (see \cite[Corollary 4.4.4]{agSpringer}). Hence, we get that the normal bundle \(\mathcal{N}_{  \mathcal{P} /    \mathcal{P} \times G }\) is the trivial bundle \(\mathcal{P} \times \mathfrak{g}. \) Then using the canonical isomorphism $\mathscr{T}_{\mathcal{P} / X} \cong \mathcal{N}_{  \mathcal{P} /    \mathcal{P} \times G }$ (see \cite{SP2}), we get the following isomorphism
\begin{equation}\label{springer}
    \mathscr{T}_{\mathcal{P} / X} \cong \mathcal{P} \times \mathfrak{g}.
\end{equation}
Note that the relative tangent bundle $\mathscr{T}_{\mathcal{P} / X}$ is \(G\)-equivariant, and the isomorphism \eqref{springer} preserves the \(G\)-action, where \(G\) acts on the Lie algebra $\mathfrak{g}$ via the adjoint representation. Since the action of \(G\) on $\mathcal{P}$ is free,  $\mathscr{T}_{\mathcal{P} / X}$ descends to a locally free sheaf on \(X\), say $\mathcal{E}$ (see \cite[Theorem 2.3]{DNarasimhan}). Using \cite[Proposition 2.23]{PerT}, we get \(\mathcal{E} \cong (p_*\mathscr{T}_{\mathcal{P} / X})^G\). Now the trivial bundle \(\mathcal{P} \times \mathfrak{g} \rightarrow \mathcal{P}\) descends to the adjoint bundle   \(\text{ad}(\mathcal{P}).\) Hence, we have 
\[			(p_*\mathscr{T}_{\mathcal{P} / X})^G \cong \text{ad}(\mathcal{P}).\]
\end{proof}

	Thus, from Proposition \ref{atseq} and Proposition \ref{atseqad} we have the following short exact sequence of sheaves
	\begin{equation}\label{atiyah}
		0 \longrightarrow \text{ad}(\mathcal{P}) \longrightarrow \mathcal{A}t(\mathcal{P})  \stackrel{\eta} \longrightarrow  \mathscr{T}_X \longrightarrow 0,
	\end{equation}
	analogous to the Atiyah sequence of bundles in the smooth case (see \cite{At}). For any open set \(U \subset X\), \(\mathcal{A}t(\mathcal{P})(U)\) has a Lie algebra structure where the Lie bracket is given by 
	\begin{equation*}
		[\tilde{\delta}_1, \tilde{\delta}_2]:=\tilde{\delta}_1 \circ \tilde{\delta}_2 - \tilde{\delta}_2 \circ \tilde{\delta}_1 \text{   for } \tilde{\delta}_1, \, \tilde{\delta}_2 \in \mathcal{A}t(\mathcal{P})(U)
	\end{equation*}
	making $\mathcal{A}t(\mathcal{P})$ a sheaf of Lie algebras. Similarly,  the sheaf \(\mathscr{T}_X\) also acquires a Lie algebra structure. To see $\eta$ is a map preserving the Lie algebra structures, it suffices to check that 
	\begin{equation*}
		\eta_U(\tilde{\delta}_1 \circ \tilde{\delta}_2)=\eta_U(\tilde{\delta}_1) \circ \eta_U(\tilde{\delta}_2) \text{ for } \tilde{\delta}_1, \, \tilde{\delta}_2 \in \mathcal{A}t(\mathcal{P})(U)
	\end{equation*} 
	
	Let \(f \in \mathcal{O}_X(U)\), then 
	\begin{equation*}
		\begin{split}
			(\eta_U(\tilde{\delta}_1) \circ \eta_U(\tilde{\delta}_2))(f) &= \eta_U(\tilde{\delta}_1)(p_*(\tilde{\delta}_2(p^{\sharp}(f) ))\\
			&=p_*\left[ \tilde{\delta}_1(p_*(\tilde{\delta}_2(p^{\sharp}(f) ) )\circ p) \right] \\
			& =p_*\left[  \tilde{\delta}_1(\tilde{\delta}_2(p^{\sharp}(f))) \right]~ (\text{using \eqref{log_con1}})\\
			& =\eta_U(\tilde{\delta}_1 \circ \tilde{\delta}_2)(f).
		\end{split}
	\end{equation*}
	This shows that the homomorphism \(\eta\) is compatible with the Lie algebra structures on the sheaves \(\mathcal{A}t(\mathcal{P})\) and \(\mathscr{T}_X\).

	Let \(D\) be a reduced Weil divisor on the normal projective variety \(X\). Then consider the subsheaf of the Atiyah sheaf \(\mathcal{A}t(\mathcal{P})\)
	\[\mathcal{A}t(\mathcal{P}) (- \text{log } D) := \eta^{-1} \left( \mathscr{T}_X(- \text{log } D)  \right) ,\] called the logarithmic Atiyah sheaf associated to the principal bundle $\mathcal{P}$. This fits into the following exact sequence of sheaves, called the logarithmic Atiyah sequence, 
	\begin{equation}\label{log_at}
		0 \longrightarrow \text{ad}(\mathcal{P}) \longrightarrow \mathcal{A}t(\mathcal{P}) (- \text{log } D)  \stackrel{\bar{\eta}} \longrightarrow  \mathscr{T}_X(- \text{log } D) \longrightarrow 0 .
	\end{equation}
	Here, \(\bar{\eta}\) is just the map $\eta$ restricted to the sheaf $\mathcal{A}t(\mathcal{P}) (- \text{log } D)$. Note that the natural Lie algebra structure on the sheaf \(\mathcal{A}t(\mathcal{P})\) induces a Lie algebra structure on \(\mathcal{A}t(\mathcal{P})(- \text{log }D)\) and $\bar{\eta}$ is compatible with the Lie algebra structures on \(\mathcal{A}t(\mathcal{P})(- \text{log }D)\) and \(\mathscr{T}_X(- \text{log } D)\). 
	
	In general, \(\text{ad}(\mathcal{P})\) is locally free, \(\mathscr{T}_X(- \text{log } D)\) is reflexive and $\mathcal{A}t(\mathcal{P}) (- \text{log } D)$ is torsion free. However, if \(X\) is smooth and \(D\) is a simple normal crossing divisor, or if \(X\) is a toric variety and \(D\) is the boundary divisor, then \(\mathcal{A}t(\mathcal{P}) (- \text{log } D)\) is locally free from \eqref{log_at}, since both \(\text{ad}(\mathcal{P})\) and \(\mathscr{T}_X(- \text{log } D)\) are locally free. 
	
	\begin{defn}
		A logarithmic connection on the principal \(G\)-bundle $\mathcal{P}$ with singular locus \(D\) is a morphism of $\mathcal{O}_X$-modules
		\begin{equation}\label{def-log-con}
			\zeta : \mathscr{T}_X(- \text{log } D) \rightarrow \mathcal{A}t(\mathcal{P}) (- \text{log } D) 
		\end{equation}
		such that \(\bar{\eta} \circ \zeta = Id_{\mathscr{T}_X(- \text{log } D)}\). We say that a logarithmic connection is integrable if the map $\zeta$ in \eqref{def-log-con} additionally preserves the Lie algebra structures.
	\end{defn}

	
	\section{Logarithmic connections on a vector bundle}\label{residue sec}
	
	If $\mathcal{P}$ is a principal bundle with structure group \(GL(r, \mathbb{K})\), then there is a one-to-one correspondence between $\mathcal{P}$ and the associated vector bundle, say \(E\). In this section, we study the relation between connections on $\mathcal{P}$ and \(E\), respectively. This was established in \cite[Proposition 3.1.1]{GP} assuming the smoothness of the variety.

	Let \(X\) be a normal projective variety with a reduced Weil divisor \(D\) such that \(\Omega^1_X(\text{log } D)\) is locally free. This happens in particular when \(X\) is a toric variety and \(D\) is the boundary divisor. Let \(\pi : E \rightarrow X\) be a vector bundle of rank \(r\) on \(X\). Let $p : \mathcal{P}\rightarrow X$ be the associated principal \(GL(r, \mathbb{K})\)-bundle. Set \(G=GL(r, \mathbb{K})\). Let $\mathcal{E}$ be the locally free sheaf of $\mathcal{O}_X$-module associated with the vector bundle \(E\). Then we define a new sheaf \(D(\mathcal{E})\) of $\mathcal{O}_X$-module as follows (cf. \cite[Section 4]{At}):
	\begin{itemize}
		\item[(i)] as a sheaf of $\mathbb{K}$-modules, \(D(\mathcal{E})=\mathcal{E} \oplus \left( \mathcal{E} \otimes_{\mathcal{O}_X} \Omega^1_X(\text{log } D)\right) \),
		\item[(ii)] the $\mathcal{O}_X$-module structure on \(D(\mathcal{E})\) is given by \[f \cdot (s \oplus \gamma)=fs \oplus(f \gamma + s \otimes df),\] 
		where \(f \), \(s\) and $\gamma$ are local sections of $\mathcal{O}_X$, $\mathcal{E}$ and \(\mathcal{E} \otimes_{\mathcal{O}_X} \Omega^1_X(\text{log } D)\), respectively. Here, \(d : \mathcal{O}_X \rightarrow \Omega^1_X\) is the universal derivation and we consider \(df\) as an element of \(\Omega^1_X(\text{log } D)\)  via the composition of natural $\mathcal{O}_X$-linear map  \(\Omega^1_X \rightarrow (\Omega^1_X)^{\vee \vee} \rightarrow \Omega^1_X(\text{log } D)\).
	\end{itemize}
	Note that \(D(\mathcal{E})\) fits into the following short exact sequence of $\mathcal{O}_X$-modules
	\begin{equation}\label{dfun1}
		0 \longrightarrow \mathcal{E} \otimes_{\mathcal{O}_X} \Omega^1_X(\text{log } D) \stackrel{\iota_2} \longrightarrow D(\mathcal{E})  \stackrel{p_1} \longrightarrow  \mathcal{E} \longrightarrow 0,
	\end{equation}
	where $\iota_2$ and \(p_1\) are the corresponding inclusion and projection maps, respectively. Hence, \(D(\mathcal{E})\) is locally free. The extension \eqref{dfun1} defines an element \(b(E) \in H^1(X, \Hom(\mathcal{E}, \mathcal{E} \otimes_{\mathcal{O}_X} \Omega^1_X(\text{log } D)))\). Note that there are canonical isomorphisms
	\begin{equation*}
		\begin{split}
			& \Hom(\mathcal{E}, \mathcal{E} \otimes_{\mathcal{O}_X} \Omega^1_X(\text{log } D)) \cong \Hom(\mathscr{T}_X(-\text{log } D), \mathcal{E}nd (\mathcal{E})) \text{ and}\\
			&	\mathcal{E}nd (\mathcal{E}) \cong ad(\mathcal{P}) ~ \text{(cf. \cite[Proposition 9]{At})}.
		\end{split}
	\end{equation*}
	
	So, we may regard \(b(E)\) as an element of \(H^1(X,  \Hom(\mathscr{T}_X(-\text{log } D), ad(\mathcal{P})))\). On the other hand,  the sequence \eqref{log_at} defines an element \(a(\mathcal{P})\) in \(H^1(X,  \Hom(\mathscr{T}_X(-\text{log } D), ad(\mathcal{P})))\), called the logarithmic Atiyah class of the principal bundle. We shall show that \(b(E)\) is essentially the same as \(a(\mathcal{P})\) by explicitly computing their corresponding representative cocycles (see \cite[Section 1]{At}).
	
	\begin{thm}\label{a=b}
		Let \(X\) be a normal projective variety. Let \(D\) be a reduced Weil divisor on \(X\) such that \(\Omega^1_X(\text{log } D)\) is locally free.	Let $\pi: E \rightarrow X$ be a vector bundle on \(X\) and $\mathcal{P}$ the associated principal \(GL(r, \mathbb{K})\)-bundle. Then the obstructions \(a(\mathcal{P})\) and \(b(E)\) are related as follows:
		\begin{equation*}
			a(\mathcal{P})=-\, b(E).
		\end{equation*}
	\end{thm} 
	
	In the rest of this subsection, we shall prove this theorem. Let \(\{U_{\sigma}\}_{\sigma}\) be a trivializing open cover of \(X\). (Note that here we are not assuming \(X\) to be a toric variety.) Let us denote \(U_{\sigma} \cap U_{\tau}\) by \(U_{\sigma \, \tau}\). Consider the following trivialization:
	\begin{equation}\label{dfun2}
		\begin{split}
			&	\phi_{\sigma} : U_{\sigma} \times \mathbb{K}^r \rightarrow E_{\sigma}:=E|_{U_{\sigma}}, \\
			&\phi_{\sigma}^{-1} \, \phi_{\tau} : U_{\sigma \, \tau} \times \mathbb{K}^r \rightarrow U_{\sigma \, \tau} \times \mathbb{K}^r, ~ (x, \lambda) \mapsto (x, g_{\sigma \, \tau}(x) \lambda), \\
			& \text{where } g_{\sigma \, \tau}: U_{\sigma \, \tau} \rightarrow G \text{ are the transition functions and } G=GL(r, \mathbb{K}).
		\end{split}
	\end{equation}
	The map \(g_{\sigma \, \tau}\) induces the following map between the coordinate rings
	\begin{equation*}
		g_{\sigma \, \tau}^{\sharp} :  \mathcal{O}(G) \rightarrow  \mathcal{O}(U_{\sigma \, \tau}).
	\end{equation*}
	The cooordinate ring of the algebraic group \(G=GL(r, \mathbb{K})\) is given by \(\mathcal{O}(G)  =\mathbb{K} [S_{ij}, \frac{1}{det(S_{ij})} \mid 1 \leq i,j \leq r]\). Hence the homomorphism $g_{\sigma \, \tau}^{\sharp}$ is determined by a matrix 
	\begin{equation}\label{transition_matrix}
		C_{\sigma \, \tau}=(c^{\sigma \, \tau}_{i j})_{r \times r} \in GL(r, \mathcal{O}(U_{\sigma \, \tau})), \text{ where } c^{\sigma \, \tau}_{i j}=	g_{\sigma \, \tau}^{\sharp}(S_{ij}).
	\end{equation}
	
	\begin{prop}\label{cocyclebP}
		A representative cocycle for \(b(E)\) in \(H^1(X,  \Hom(\mathscr{T}_X(-\text{log } D),  \mathcal{E}nd (\mathcal{E})))\) is given by \(\left\lbrace \left( \delta \mapsto - d(C_{ \sigma \, \tau}) C_{\tau \, \sigma}(\delta)\right) \right\rbrace \), $\delta \in \mathscr{T}_{U_{\sigma \, \tau}}(-\text{log } D)(U_{\sigma \, \tau})$.
	\end{prop}
	
	\begin{proof}
		Write \(U=U_{\sigma \, \tau}\) for ease of notation. The induced sheaf homomorphisms of the bundle maps in \eqref{dfun2} are given as follows:
		\begin{equation}\label{dfun1.1}
			\begin{split}
				& \boldsymbol{\phi}_{ \, \sigma} : \mathcal{O}_{U}^{\oplus r} \rightarrow \mathcal{E}_{\sigma}:=\mathcal{E}|_{U_{\sigma}}, \\
				& \boldsymbol{\phi}_{\, \sigma}^{-1} \boldsymbol{\phi}_{\, \tau} : \mathcal{O}_{U}^{\oplus r} \rightarrow \mathcal{O}_{U}^{\oplus r} , \, \underline{f}=(f_1, \ldots, f_r)^t \mapsto C_{\sigma \, \tau} \underline{f}.
			\end{split}
		\end{equation}
		Now, the universal derivation \(d:\mathcal{O}_{U_{\sigma}} \rightarrow \Omega^1_{U_{\sigma}}\) induces a map, also denoted by \(d\), defined by \[d:\mathcal{O}_{U_{\sigma}}^{\oplus r} \rightarrow \mathcal{O}_{U_{\sigma}}^{\oplus r} \otimes \Omega^1_{U_{\sigma}}(\text{log } D).\]
		Define \(d_{\sigma}\) as the composite of the following maps:
		\begin{equation}\label{dfun2.2}
			\mathcal{E}_{\sigma} \stackrel{\boldsymbol{\phi}_{\, \sigma}^{-1}} \longrightarrow \mathcal{O}_{U_{\sigma}}^{\oplus r} \stackrel{d} \longrightarrow \mathcal{O}_{U_{\sigma}}^{\oplus r} \otimes \Omega^1_{U_{\sigma}}(\text{log } D) \stackrel{\boldsymbol{\phi}_{\, \sigma} \otimes id} \longrightarrow \mathcal{E}_{\sigma} \otimes \Omega^1_{U_{\sigma}}(\text{log } D).
		\end{equation}
		This enables us to define a local splitting of the exact sequence \eqref{dfun1} as $\mathcal{O}_X$-modules
		\begin{equation*}
			\begin{split}
				\boldsymbol{\psi}_{\, \sigma} : \mathcal{E}_{\sigma} & \rightarrow D(\mathcal{E})_{\sigma}, \\
				s & \mapsto s \oplus d_{\sigma}(s),
			\end{split}
		\end{equation*}
		where \(s \) is a local section of $\mathcal{E}_{\sigma}$. Then define elements \(b_{\sigma \, \tau} \in H^0(U, \Hom(\mathcal{E}, \mathcal{E} \otimes \Omega^1_{X}(\text{log } D)))\) by 
		\begin{equation}\label{dfun2.3}
			b_{\sigma \, \tau}(s)=	\boldsymbol{\psi}_{\, \tau}(s)-	\boldsymbol{\psi}_{\, \sigma}(s)=d_{\tau}(s)-d_{\sigma}(s).
		\end{equation}
		Note that \(\{b_{\sigma \, \tau}\}\) is a representative cocycle for \(b(E)\) (cf. \cite[Pages 184 - 185]{At}). Next, we give an explicit description of the composition of the following maps
		\begin{equation}\label{dfun2.4}
			(\boldsymbol{\phi}_{\, \sigma} \otimes id)^{-1} \circ b_{\sigma \, \tau} \circ \boldsymbol{\phi}_{\, \sigma} |_{U} : \mathcal{O}_{U}^{\oplus r} \rightarrow \mathcal{O}_{U}^{\oplus r} \otimes \Omega^1_{U}(\text{log } D) \cong \Omega^1_{U}(\text{log } D)^{\oplus  r}.
		\end{equation}
		Let \(\underline{f}=(f_1, \ldots, f_r)^t \in \mathcal{O}_{U}^{\oplus r}\), then we have
		\begin{equation}\label{dfun1.2}
			\begin{split}
				&	(\boldsymbol{\phi}_{\, \sigma} \otimes id)^{-1} \circ b_{\sigma \, \tau} \circ \boldsymbol{\phi}_{\, \sigma}(\underline{f})\\
				& = (\boldsymbol{\phi}_{\, \sigma} \otimes id)^{-1} \circ d_{\tau} \circ \boldsymbol{\phi}_{\, \sigma}(\underline{f})-(\boldsymbol{\phi}_{\, \sigma} \otimes id)^{-1} \circ d_{\sigma} \circ \boldsymbol{\phi}_{\, \sigma}(\underline{f}) ~ ~ (\text{by \eqref{dfun2.3}})\\
				& =(\boldsymbol{\phi}_{\, \sigma} \otimes id)^{-1} \circ (\boldsymbol{\phi}_{\, \tau} \otimes id) \circ d \circ \boldsymbol{\phi}_{\, \tau}^{-1} \circ \boldsymbol{\phi}_{\, \sigma}(\underline{f})-(\boldsymbol{\phi}_{\, \sigma} \otimes id)^{-1} \circ d_{\sigma} \circ \boldsymbol{\phi}_{\, \sigma}(\underline{f}) ~ ~ (\text{by \eqref{dfun2.2}})\\
				&= (\boldsymbol{\phi}_{\, \sigma} \otimes id)^{-1} \circ (\boldsymbol{\phi}_{\, \tau} \otimes id) \circ d\left( C_{\tau \, \sigma} \underline{f}\right) -d(\underline{f}) ~ ~ (\text{using \eqref{dfun1.1}})\\
				&=( \phi_{\sigma}^{-1} \, \phi_{\tau} \otimes id) \circ \left( d \left( C_{\tau \, \sigma}\right)  \underline{f} + C_{\tau \, \sigma} d \underline{f} \right) -d(\underline{f}) ~ ~ (\text{since } d\left( C_{\tau \, \sigma} \underline{f}\right)=d \left( C_{\tau \, \sigma}\right)  \underline{f} + C_{\tau \, \sigma} \, d \underline{f})\\
				&=C_{\sigma \, \tau}  \, d\left( C_{\tau \, \sigma}\right) \underline{f} +  C_{\sigma \, \tau} \, C_{\tau \, \sigma} d (\underline{f})-d(\underline{f})~ ~ (\text{again by \eqref{dfun1.1}})\\
				&=C_{\sigma \, \tau}  \, d\left( C_{\tau \, \sigma}\right) \underline{f} ~ ~(\text{since } C_{\sigma \, \tau} \, C_{\tau \, \sigma}=Id)\\
				&= -d \left( C_{\sigma \, \tau}\right) C_{\tau \, \sigma} (\underline{f})~ ~(\text{since } C_{\sigma \, \tau}d(C_{\tau \, \sigma}) + d \left( C_{\sigma \, \tau}\right) C_{\tau \, \sigma}=0).
			\end{split}
		\end{equation}

		Thus, we have
		\begin{equation*}
			\begin{split}
				& b_{\sigma \, \tau}=(\boldsymbol{\phi}_{\, \sigma} \otimes id) \circ (-d \left( C_{\sigma \, \tau}\right) C_{\tau \, \sigma}) \circ \boldsymbol{\phi}_{\, \sigma}^{-1} |_{U}: \mathcal{E}|_U \rightarrow \mathcal{E}|_U \otimes \Omega^1_{U}(\text{log } D).
			\end{split}
		\end{equation*}
		Under the following  canonical identification \[\text{Hom}_{\mathcal{O}(U)}(\mathcal{E}(U), \mathcal{E}(U) \otimes \Omega^1_{U}(\text{log } D)(U)) \cong \text{Hom}_{\mathcal{O}(U)}(\mathscr{T}_U(-\text{log } D)(U), \text{End}(\mathcal{E}(U))),\] 
		the element \(b_{\sigma \, \tau}\) corresponds to the map 
		\begin{equation*}
			\begin{split}
				\widetilde{b}_{\sigma \, \tau}: \mathscr{T}_U(-\text{log } D)(U)  & \rightarrow \text{End}(\mathcal{E}(U))  \cong M(r, A),\\
				\delta & \mapsto -  d(C_{ \sigma \,\tau}) C_{\tau \, \sigma} (\delta),
			\end{split}
		\end{equation*}
		where \(\delta \in \mathscr{T}_U(-\text{log } D)(U)\) and \(A=\mathcal{O}(U)\).

	\end{proof}
	
	We now calculate the obstruction \(a(\mathcal{P})\) explicitly in terms of the transition functions of the principal bundle $\mathcal{P}$ following \cite[Section 3]{At}. The associated principal bundle $\mathcal{P}$ has a local trivialization induced from the local trivialization given in \eqref{dfun2}:
	\begin{equation}\label{dfun3}
		\begin{split}
			&	\varphi_{\sigma} : U_{\sigma} \times G \rightarrow \mathcal{P}_{\sigma}:=\mathcal{P}|_{U_{\sigma}} \\
			&\varphi_{\sigma}^{-1} \, \varphi_{\tau} : U_{\sigma \, \tau} \times G \rightarrow U_{\sigma \, \tau} \times G, (x, \lambda) \mapsto (x, g_{\sigma \, \tau}(x) \lambda), 
		\end{split}
	\end{equation}
	\(\text{where } g_{\sigma \, \tau}: U_{\sigma \, \tau} \rightarrow G \text{ are the transition functions}.\)
	The isomorphism 
	\[\varphi_{\sigma} : U_{\sigma} \times G \stackrel{\cong} \rightarrow \mathcal{P}_{\sigma}:=\mathcal{P}|_{U_{\sigma}}\] induces a \(G\)-equivariant isomorphism:
	\begin{equation*}
		\widehat{\varphi}_{\sigma} :  \text{Der}_{\mathbb{K}} ( \mathcal{O}(U_{\sigma}) \otimes \mathcal{O}(G) ) \rightarrow \text{Der}_{\mathbb{K}} (\mathcal{O}(\mathcal{P}_{\sigma}) ).
	\end{equation*}
	Taking invariants, we get
	\begin{equation*}
		\widehat{\varphi}_{\sigma} : ( \text{Der}_{\mathbb{K}}( \mathcal{O}(U_{\sigma}) \otimes \mathcal{O}(G) ))^G \rightarrow (\text{Der}_{\mathbb{K}} (\mathcal{O}(\mathcal{P}_{\sigma})))^G.
	\end{equation*}
	Note that \[  \text{Der}_{\mathbb{K}}( \mathcal{O}(U_{\sigma}) ) \subset ( \text{Der}_{\mathbb{K}}( \mathcal{O}(U_{\sigma}) \otimes \mathcal{O}(G) ))^G,\] where we consider \(\delta \in  \text{Der}_{\mathbb{K}}( \mathcal{O}(U_{\sigma})) \) as an element of \(\text{Der}_{\mathbb{K}}( \mathcal{O}(U_{\sigma}) \otimes \mathcal{O}(G) ))^G\) in the following way:
	\begin{equation*}
		\alpha \otimes \beta \mapsto \delta(\alpha) \otimes \beta, \text{ where } \alpha \in \mathcal{O}(U_{\sigma}) \text{ and } \beta \in \mathcal{O}(G).
	\end{equation*}
	Consider the restriction of \(\widehat{\varphi}_{\sigma} \) to $\text{Der}_{\mathbb{K}}( \mathcal{O}(U_{\sigma}) )$ given by
	\begin{equation*}
		a_{\sigma}: \text{Der}_{\mathbb{K}}( \mathcal{O}(U_{\sigma}) ) \rightarrow  (\text{Der}_{\mathbb{K}} (\mathcal{O}(\mathcal{P}_{\sigma})))^G.
	\end{equation*}
	Then \(a_{\sigma}\) defines a local splitting of \eqref{atiyah}. Further restricting this map to \(\mathscr{T}_{U_{\sigma}}( - \text{ log }D)\), we get
	\begin{equation}\label{obs1}
		a_{\sigma} : \mathscr{T}_{U_{\sigma}}( - \text{log }D)(U_{\sigma})  \rightarrow \mathcal{A}t(\mathcal{E}) (- \text{log } D)(U_{\sigma}),
	\end{equation}
	which gives a local splitting of logarithmic Atiyah sequence \eqref{log_at}. Set \(U=U_{\sigma \, \tau}\) and define 
	\begin{equation*}
		\begin{split}
			a_{\sigma\tau}: \mathscr{T}_{U}( - \text{log }D)(U) & \rightarrow  \mathcal{A}t(\mathcal{E}) (- \text{log } D)(U), \\
			a_{\sigma\tau} (\delta):&=	a_{\tau}(\delta)-  	a_{\sigma}(\delta) = \widehat{\varphi}_{\tau}(\delta)- \widehat{\varphi}_{\sigma}(\delta) ~ (\text{using \eqref{obs1}}).
		\end{split}
	\end{equation*} 
	Then \(\{a_{\sigma \, \tau}\}\) is a representative cocycle for \(a(\mathcal{P})\). Recall the definition of \(C_{\sigma \, \tau}\) from \eqref{transition_matrix}.
	
	\begin{prop}\label{cocycleAP}
		For \(\delta \in \mathscr{T}_U(-\text{log } D)(U)\), we have 
		\begin{equation*}
			\widehat{\varphi}_{\sigma}^{-1}(a_{\sigma\tau} (\delta))=\left( d(C_{\sigma \, \tau}) C_{\tau \, \sigma} \right) (\delta).
		\end{equation*}
		
	\end{prop}
	
	\begin{proof}

		Observe that 
		\begin{equation}\label{isopart1}
			\widehat{\varphi}_{\sigma}^{-1}(a_{\sigma\tau} (\delta))=(\widehat{\varphi}_{\sigma}^{-1} \circ \widehat{\varphi}_{\tau})(\delta)-\delta=\widehat{(\varphi_{\sigma}^{-1} \circ \varphi_{\tau})}(\delta) - \delta.
		\end{equation}
		
		Henceforth, we use \(U=U_{\sigma \, \tau}, \, \varphi=\varphi_{\sigma}^{-1} \, \varphi_{\tau}\), \(A=\mathcal{O}(U)\) and \(B=\mathcal{O}(G)\) for ease of notation. We write \[\varphi= (\text{pr}_1, \, \mu)\] where
		\begin{equation*}
			\text{pr}_1: U \times G \ \rar \ U 
		\end{equation*}
		is the first projection and
		\begin{equation*}
			\mu: U \times G \ \rar \ G 
		\end{equation*}
		is the following composition 
		\begin{equation*}
			U \times G \ \stackrel{g_{\sigma \, \tau} \times \text{ id}}{\longrightarrow} \ G \times G \ \stackrel{m}{\rar} \ G.
		\end{equation*}
		Thus, at the level of rings, we have the following maps
		\begin{equation}\label{ring}
			\begin{split}
				& \text{pr}_1^{\sharp} : A \rar A \otimes_{\mathbb{K}} B \text{ given by } \alpha \mapsto \alpha \otimes 1,\\
				& \mu^{\sharp}: B \rar A \otimes_{\mathbb{K}} B \text{ given by } S_{p \, q } \mapsto \sum_k g_{\sigma \,  \tau}^{\sharp}(S_{p \, k}) \otimes S_{k \, q} \text{ and}\\
				& \varphi^{\sharp}: A \otimes_{\mathbb{K}} B \rar A \otimes_{\mathbb{K}} B \text{ given by } \alpha \otimes \beta \mapsto \text{pr}_1^{\sharp}(\alpha) \mu^{\sharp}(\beta).
			\end{split}
		\end{equation}
		Thus, we get the map
		\begin{equation*}
			\begin{split}
				\widehat{\varphi}:\text{Der}_{\mathbb{K}}(A \otimes_{\mathbb{K}} B) \rar \text{Der}_{\mathbb{K}}(A \otimes_{\mathbb{K}} B) \text { given by } \delta \mapsto \delta \circ \varphi^{\sharp}.
			\end{split}
		\end{equation*}
		Since \(\varphi\) and hence \(\widehat{\varphi}\) is \(G\)-equivariant, it induces a map of invariants
		\begin{equation}\label{obs2}
			\widehat{\varphi}:\text{Der}_{\mathbb{K}}(A \otimes_{\mathbb{K}} B)^G \ \rar \ \text{Der}_{\mathbb{K}}(A \otimes_{\mathbb{K}} B)^G. 
		\end{equation}
		We have an isomorphism 
		\begin{equation*}
			\begin{split}
				\text{Der}_{\mathbb{K}}(A \otimes_{\mathbb{K}} B)^G &\cong \text{Der}_{\mathbb{K}}(A, \, A \otimes_{\mathbb{K}} B)^G \oplus \text{Der}_{\mathbb{K}}(B, \, A \otimes_{\mathbb{K}} B)^G  \text{given by }\\
				\delta &\mapsto \delta_1 \oplus \delta_2, \text{ where }\delta_1(\alpha)=\delta(\alpha \otimes 1) \text{ and }\delta_2(\beta)= \delta (1 \otimes \beta).
			\end{split}
		\end{equation*}
		The inverse map is given as follows. Given $\delta_1 \in \text{Der}_{\mathbb{K}}(A, \, A \otimes_{\mathbb{K}} B)^G$ and $\delta_2 \in \text{Der}_{\mathbb{K}}(B, \, A \otimes_{\mathbb{K}} B)^G$, we define $\delta \in 	\text{Der}_{\mathbb{K}}(A \otimes_{\mathbb{K}} B)^G$ using the Leibniz rule
		\begin{equation*}
			\delta(\alpha \otimes \beta)=\delta_1(\alpha) \, \beta + \alpha \, \delta_2(\beta),
		\end{equation*}
		where $\alpha \in A$ and $\beta \in B$. Let \(\delta \in \mathscr{T}_{U}( - \text{log }D)(U)  \subseteq \text{Der}_{\mathbb{K}}(A)\) and write \(\widehat{\varphi}(\delta)=\delta_1 \oplus \delta_2\). We wish to determine \(\delta_1\) and $\delta_2$. From \eqref{ring} and \eqref{obs2}, we have
		\begin{equation*}
			\begin{split}
				\delta_1(\alpha)=\widehat{\varphi}(\delta)(\alpha \otimes 1)= (\delta \circ \varphi^{\sharp}) (\alpha \otimes 1)= (\delta \circ \text{pr}_1^{\sharp})(\alpha)=  \delta(\alpha \otimes 1)= \delta (\alpha) \otimes 1. 
			\end{split}
		\end{equation*}
		This shows that 
		\begin{equation*} \label{obs3}
			\delta_1=\delta \in \mathscr{T}_{U}( - \text{log }D)(U). 
		\end{equation*} 
		On the other hand, note that 
		\begin{equation}
			\begin{split}
				\delta_2(S_{p \, q})=\widehat{\varphi}(\delta)(1 \otimes S_{p \, q})= (\delta \circ \varphi^{\sharp}) (1 \otimes S_{p \, q})= \delta ( \mu^{\sharp}(S_{p \, q})).
			\end{split}
		\end{equation}
		Observe that
		\begin{equation}\label{matrix}
			\begin{split}
				\text{Der}_{\mathbb{K}}(B, \, A \otimes_{\mathbb{K}} B) ^G &\cong (\text{Der}_{\mathbb{K}}(B) \otimes_B (A \otimes_{\mathbb{K}} B)) ^G \, (\text{using the universal property of derivations.}) \\
				&\cong  (\text{Der}_{\mathbb{K}}(B) \otimes_{\mathbb{K}} A) ^G \cong \text{Der}_{\mathbb{K}}(B)^G \otimes _{\mathbb{K}} A^G  \cong \mathfrak{g} \otimes_{\mathbb{K}} A \cong M(r, A).
			\end{split}
		\end{equation}
		Thus, \(\delta_2 \in \text{Der}_{\mathbb{K}}(B, \, A \otimes_{\mathbb{K}} B) ^G \) can be identified with a matrix \((f_{p q})_{r \times r}\) with entries in \(A\).  Then using \eqref{isopart1} and \eqref{obs3}together with Lemma \ref{matrixfpq} to follow, we have 
		\begin{equation}\label{obs5}
			\widehat{\varphi}_{\sigma}^{-1}(a_{\sigma\tau} (\delta))= \delta_2=\left( d(C_{\sigma \, \tau}) C_{\tau \, \sigma} \right) (\delta).
		\end{equation}
		Hence the Proposition \ref{cocycleAP} follows.
	\end{proof}

	\begin{lemma}\label{matrixfpq}
		With the notations as above, for \(\delta \in \mathscr{T}_U(-\text{log } D)(U)\) we have 			
		$$(f_{p q})_{r \times r}=\left( d(C_{\sigma \, \tau}) C_{\tau \, \sigma} \right) (\delta).$$ 
	\end{lemma}
	
	\begin{proof}
		
		Fix a basis \(\{\omega_{k \, l} ~|~ 1 \leq k, l \leq r\}\) of $\mathfrak{g}=\text{Der}_{\mathbb{K}}(B_{1_G}, \mathbb{K})$, where \(B_{1_G}\) denote the local ring at the identity and for any \(f \in B_{1_G}\),
		\begin{equation}\label{basis}
			\begin{split}
				\omega_{k \, l}(f)= \frac{\partial}{\partial S_{k \, l}}(f)(1_G).
			\end{split}
		\end{equation}  Take \(x \in U\). We will also denote by \(x\) the corresponding maximal ideal in \(A\). Note that the map 
		\begin{equation*}
			\begin{split}
				\iota : U \rightarrow U \times G, \,  x \mapsto (x,1_G)
			\end{split}
		\end{equation*}
		induces an isomorphism \[U \cong (U \times G)/G.\] 
		We get a point derivation on \(G\) using the following composition
		\begin{equation}\label{compo.p21}
			B_{1_G} \stackrel{( \cdot g_{\sigma \, \tau}(x)^{-1})^{\sharp}}{\longrightarrow} B_{g_{\sigma \, \tau}(x)} \stackrel{(\mu^{\sharp})_{(x,1)}}\longrightarrow (A \otimes B)_{(x,1)} \stackrel{(\iota^{\sharp})_x} \longrightarrow A_x \stackrel{\delta_x}\longrightarrow \mathbb{K},
		\end{equation}
		where the first map is the ring map corresponding to the translation map of \(G\) by \(g_{\sigma \, \tau}(x)^{-1}\) and the last map is given by \(\frac{a}{c} \mapsto \frac{(c \, \delta(a)+a \, \delta(c))(x)}{c(x)^2}\). The composite in \eqref{compo.p21}, say $\tilde{\delta}_x$ gives rise to an element of $\mathfrak{g}$, hence we can write 
		\begin{equation*}
			\tilde{\delta}_x= \sum f_{pq}(x) \, \omega_{pq}.
		\end{equation*}
		
		Applying \(S_{pq}\) on both sides of the above equation, and using \eqref{basis}, we have 
		\begin{equation}\label{coordinates}
			f_{pq}(x)=\tilde{\delta}_x(S_{p \, q})=\delta_x \circ (\iota^{\sharp})_x \circ (\mu^{\sharp})_{(x,1_G)}(S_{pq} \cdot g_{\sigma \, \tau}(x)).
		\end{equation}
		Now, for any \(h \in G\), we have that 
		\begin{equation*}
			\begin{split}
				&	(S_{pq} \cdot g_{\sigma \, \tau}(x))(h)= S_{pq} (h \cdot g_{\sigma \, \tau}(x)^{-1})= S_{pq} (h \cdot g_{\tau \, \sigma}(x))\\
				&= \sum_k h_{pk} \, (g_{\tau \, \sigma}(x))_{kq}=\sum_k S_{pk}(h) \, S_{kq}(g_{\tau \, \sigma}(x))=\sum_k S_{pk}(h) \, g_{\tau \, \sigma}^{\sharp}(S_{kq})(x).
			\end{split}
		\end{equation*}
		This shows that
		\begin{equation*}
			S_{pq} \cdot g_{\sigma \, \tau}(x)= \sum_k S_{pk} \, g_{\tau \, \sigma}^{\sharp}(S_{kq})(x).
		\end{equation*}
		From \eqref{coordinates}, we have that 
		\begin{equation*}
			\begin{split}
				f_{pq}(x)&=\delta_x \circ (\iota^{\sharp})_x \circ (\mu^{\sharp})_{(x,1_G)}\left( \sum_k S_{pk} \, g_{\tau \, \sigma}^{\sharp}(S_{kq})(x)\right) \\
				&=\delta_x \circ (\iota^{\sharp})_x\left( \sum_k \left[ \sum_l g_{\sigma \, \tau}^{\sharp}(S_{pl}) \otimes S_{l \, k}\right]  \, g_{\tau \, \sigma}^{\sharp}(S_{kq})(x)\right) ~ ~ (\text{by } \eqref{obs5}) \\
				&=\delta_x \left( \sum_k \left[ \sum_l g_{\sigma \, \tau}^{\sharp}(S_{pl}) \otimes S_{l \, k}(1_G)\right]  \, g_{\tau \, \sigma}^{\sharp}(S_{kq})(x)\right) \\
				&=\delta_x \left( \sum_k g_{\sigma \, \tau}^{\sharp}(S_{pk})  \, g_{\tau \, \sigma}^{\sharp}(S_{kq})(x)\right) ~ (\text{since } \mu \circ \iota=g_{\sigma \, \tau})\\
				&=\sum_k \delta(g_{\sigma \, \tau}^{\sharp}(S_{pk}))(x) \, g_{\tau \, \sigma}^{\sharp}(S_{kq})(x).
			\end{split}
		\end{equation*}
		Thus, we have \(f_{pq}=\sum_k \delta(g_{\sigma \, \tau}^{\sharp}(S_{pk})) \, g_{\tau \, \sigma}^{\sharp}(S_{kq}).\) Now, since \(\Omega_X^1(\text{log }D)\) is dual to  \( \mathscr{T}_X(- \text{log } D)\), we have
		\begin{equation*}
			\delta(g_{\sigma \, \tau}^{\sharp}(S_{pk}))=dg_{\sigma \, \tau}^{\sharp}(S_{pk})(\delta).
		\end{equation*}
		Thus using \eqref{transition_matrix}, for \(\delta \in \mathscr{T}_U(-\text{log } D)(U)\), we get a matrix with entries in \(A\), given by 
		\begin{equation*} 
			(f_{p q})_{r \times r}=\left( d(C_{\sigma \, \tau}) C_{\tau \, \sigma} \right) (\delta).
	\end{equation*} \end{proof}

	\begin{proof}[Proof of Theorem \ref{a=b}]
		The statement of Theorem \ref{a=b} is now an immediate consequence of Proposition \ref{cocyclebP} and Proposition \ref{cocycleAP}.
	\end{proof}

	\begin{cor}\label{equivalence}
		With the above notations, the principal bundle \(\mathcal{P} \) admits a logarithmic connection singular over \(D\) if and only if  there is a \(\mathbb{K}\)-linear sheaf homomorphism 
		\begin{equation*}
			\nabla : E \longrightarrow E \otimes_{\mathcal{O}_X} \Omega^1_X(\text{log } D)
		\end{equation*}
		satisfying the Leibniz rule 
		\begin{equation*}
			\nabla(f \, s )=f \, \nabla(s) + s \otimes df,
		\end{equation*}
		for any open subset \(U\) of \(X\) and all sections \(f \in \mathcal{O}_X(U) \) and \(s \in \Gamma(U, E)\).
	\end{cor}

 The following proposition shows that there is a natural induced connection  on the vector bundle \(E\) given a connection on its associated principal bundle \(\mathcal{P} \).

	\begin{prop}\label{equivalence1} Let \(X\) be a normal projective variety and \(D\) be a reduced Weil divisor such that the logarithmic cotangent sheaf $\Omega^1_X(\text{log } D)$ is locally free. Let \(\pi : E \rightarrow X\) be a vector bundle of rank \(r\) on \(X\). Let $p : \mathcal{P}\rightarrow X$ be the associated principal \(GL(r, \mathbb{K})\)-bundle.
		If the principal bundle \(\mathcal{P} \) admits a logarithmic connection singular over \(D\), then we get a \(\mathbb{K}\)-linear sheaf homomorphism 
		\begin{equation*}
			\nabla : E \longrightarrow E \otimes_{\mathcal{O}_X} \Omega^1_X(\text{log } D)
		\end{equation*}
		induced from the connection on its frame bundle. The homomorphism satisfies the Leibniz rule 
		\begin{equation*}
			\nabla(f \, s )=f \, \nabla(s) + s \otimes df,
		\end{equation*}
		for any open subset \(U\) of \(X\) and all sections \(f \in \mathcal{O}_X(U) \) and \(s \in \Gamma(U, E)\).
	\end{prop}
	
	\begin{proof} 
		Let \(\zeta\) be a logarithmic connection of \(\mathcal{P}\) singular over \(D\) as defined in \eqref{def-log-con}. The pullback sheaf \(p^* E\) is trivial, in fact, there is a canonical isomorphism 
		\begin{equation*}
			\Phi  : p^* E \longrightarrow \mathcal{O}_{\mathcal{P}}^{\oplus r}
		\end{equation*}
		of \(\mathcal{O}_{\mathcal{P}}\)-modules defined as follows (see \cite[proof of Proposition 2.1]{BMS}). The total space of the vector bundle \(p^* E\) is given as follows:
		\begin{equation*}
			p^* E = \{(e, y) \in \mathcal{P} \times E ~ | ~ p(e)=\pi(y)\}.
		\end{equation*}
		Let $\tilde{s}$ be a local section of \(p^* E\) over an open subset \(V \subset \mathcal{P}\). Let \(e \in V\) and set \(x=p(e)\). We can consider \(e=(e_1, \ldots, e_r) \) as an ordered basis of the fiber \(E(x)\). Then, we can write
		\begin{equation*}
			\tilde{s}(e)=(e, \tilde{f}_1(e)e_1 + \ldots + \tilde{f}_r(e)e_r),
		\end{equation*}
		where \(\tilde{f}_1, \ldots, \tilde{f}_r \in \mathcal{O}_{\mathcal{P}}(V)\). Then, we have
		\begin{equation*}
			\Phi(\tilde{s})=(\tilde{f}_1, \ldots, \tilde{f}_r).
		\end{equation*}
		Set \(G=GL(r, \mathbb{K})\). Then \(G\) acts on the right of \(p^* E\)  by \( (e,y) \cdot g=(eg, y)\). Let \(V \subset \mathcal{P}\) be a \(G\)-invariant open subset, then there is an action of \(G\) on the space of sections $\Gamma(V, p^* E)$. Let  $\tilde{s}$ be a local section of \(p^* E\) over an open subset \(V \) and \(g \in G\), then define
		\begin{equation*}
			\begin{split}
				(\tilde{s} \cdot g)(e)=\tilde{s}(e \cdot g^{-1}) \cdot g,
			\end{split}
		\end{equation*}
		where \(e \in V\). Next, we observe that $\Phi$ is \(G\)-equivariant. Note that
		\begin{equation*}
			\begin{split}
				(\tilde{s} \cdot g)(e)=(e \cdot g, \tilde{f}_1(e \cdot g^{-1})e_1 + \ldots + \tilde{f}_r(e \cdot g^{-1} )e_r).
			\end{split}
		\end{equation*}
		Hence, \(\Phi(\tilde{s} \cdot g)=(\tilde{f}_1 \cdot g, \ldots, \tilde{f}_r \cdot g).=\Phi(\tilde{s}) \cdot g\).
		
		Next, we define a map 
		\begin{equation*}
			\widetilde{\nabla} : E \otimes_{\mathbb{K}} \mathscr{T}_X(-\text{log } D) \longrightarrow E
		\end{equation*}
		as follows: Let \(s\) be a local section of \(E\) over an affine open subset \(U=\text{Spec}(B) \subset X\). Let \(p^{-1}(U)=\text{Spec}(A) \subset \mathcal{P}\). Then \((p^*E)(p^{-1}(U))=E(U) \otimes_B A\). So \(s \otimes 1\) gives a local section of \(p^*E\) over the affine open set  \(p^{-1}(U)=\text{Spec}(A)\). Let 
		\begin{equation}\label{img under Phi}
			\Phi(s \otimes 1)=(\tilde{f}_1, \ldots, \tilde{f}_r) \in \mathcal{O}_{\mathcal{P}}^{\oplus r}.
		\end{equation}
		Let $\delta$ be a local section of \(\mathscr{T}_X(-\text{log } D) \) over \(U\). Then set 
		\begin{equation}\label{nabla1}
			\widetilde{\nabla}(s \otimes \delta)=\Phi^{-1}((\zeta(\delta)(\tilde{f}_1), \ldots, \zeta(\delta)(\tilde{f}_r))).
		\end{equation}
		Note that since \(s \otimes 1 \) is \(G\)-invariant and $\Phi$ is \(G\)-equivariant, in fact we have 
		\begin{equation*}
			\Phi^{-1}((\zeta(\delta)(g_1), \ldots, \zeta(\delta)(g_r))) \in (E(U) \otimes_B A)^G=E(U) \otimes_B A^G=E(U),
		\end{equation*} where the last equality uses the isomorphism \(p_*:A^G \, \rar \, B\). Let \(c \in \mathbb{K} \), \(s \in E(U)\) and \(f \in \mathcal{O}_X(U)\) for any open set \(U \subset X\). Since, the map $\Phi$ is $\mathcal{O}_{\mathcal{P}}$-linear and the connection $\zeta$ is $\mathcal{O}_X$-linear, we have
  \begin{equation}\label{cU-linear}
      \widetilde{\nabla}(cs \otimes \delta)=c \, \widetilde{\nabla}(s \otimes \delta)  \text{ and }  \widetilde{\nabla}(s \otimes f \delta)=f \, \widetilde{\nabla}(s \otimes \delta).
  \end{equation}
 Also, note that \[fs \otimes 1=(p^{\sharp}f) (s \otimes 1).\] Then, $$\Phi(fs \otimes 1)=p^{\sharp}f \Phi(s \otimes 1)=((p^{\sharp}f) \tilde{f}_1, \ldots, (p^{\sharp}f) \tilde{f}_r).$$ Thus, we have
		\begin{equation}\label{leib0}
			\begin{split}
				\widetilde{\nabla}(fs \otimes \delta) & =\Phi^{-1}((\zeta(\delta)((p^{\sharp}f) \tilde{f}_1), \ldots, \zeta(\delta)((p^{\sharp}f ) \tilde{f}_r))).
			\end{split}
		\end{equation}
		Note that for \(i=1, \ldots, r\), we have
		\begin{equation}\label{leib2}
			\zeta(\delta)((p^{\sharp}f ) \tilde{f}_i)=(p^{\sharp}f ) \, \zeta(\delta)(\tilde{f}_i) + \tilde{f}_i \, \zeta(\delta)((p^{\sharp}f )).
		\end{equation}
		Also, observe that 
		\begin{equation*}
			\eta(\zeta(\delta))(f)=p_*(\zeta(\delta)(p^{\sharp}(f) )).
		\end{equation*}
		Since $\zeta$ is a connection, i.e. \(\eta \circ \zeta= \text{Id}_{\mathscr{T}_X(-\text{log }D)}\), we have 
		\begin{equation}\label{compat}
			p_*\left[ \zeta(\delta)(p^{\sharp}f ) \right]= \delta(f).
		\end{equation}
		Thus 	from \eqref{leib0}, we have
		\begin{equation}\label{leibniz1}
			\begin{split}
				\widetilde{\nabla}(fs \otimes \delta) & =\Phi^{-1} \left[ (p^{\sharp}f) (\zeta(\delta)(\tilde{f}_1), \ldots, \zeta(\delta)(\tilde{f}_r))+ \zeta(\delta)((p^{\sharp}f )) (\tilde{f}_1, \ldots, \tilde{f}_r) \right] \ (\text{by \eqref{leib2}}) \\
				&= (p^{\sharp}f) \Phi^{-1} ((\zeta(\delta)(\tilde{f}_1), \ldots, \zeta(\delta)(\tilde{f}_r))) + \zeta(\delta)((p^{\sharp}f )) \, \Phi^{-1}(\tilde{f}_1, \ldots, \tilde{f}_r)\\
				&= f \, \widetilde{\nabla}(s \otimes \delta) + \delta(f) s \ (\text{by \eqref{nabla1} and \eqref{compat}}).
			\end{split}
		\end{equation}

	 Now define 
		\begin{equation}\label{nabla2}
			\nabla : E \longrightarrow E \otimes_{\mathcal{O}_X} \Omega^1_X(\text{log } D)
		\end{equation}
  induced from the map $\widetilde{\nabla}$ as follows: Let $U$ be an open affine subset of $X$ and $s \in E(U)$. Consider the $\mathcal{O}_X(U)$-linear map 
  \begin{equation*}
      \widetilde{\nabla}^U_s ~:~ \mathscr{T}_X(-\text{log } D)(U) \rightarrow E(U), \text{ given by } \delta \mapsto \widetilde{\nabla}^U(s \otimes \delta).
  \end{equation*}
		Then define
  \[\nabla^U(s)=\Psi^U(\widetilde{\nabla}^U_s),\] where $\Psi : \Hom(\mathscr{T}_X(-\text{log } D), E) \cong \Omega^1_X(\text{log } D) \otimes E$ is the canonical isomorphism of sheaves. Using \eqref{cU-linear} and \eqref{leibniz1}, it follows immediately that $\nabla$ is a $\mathbb{K}$-linear map satisfying the Leibniz rule.
	\end{proof}

     	\section{Logarithmic connection on associated bundles}\label{sec 4.1}
      In this section, we deal with logarithmic connections induced on associated bundles of a principal bundle. This has been studied for holomorphic principal bundles on compact Riemann surfaces and smooth projective varieties over complex numbers in \cite{BDPS} and \cite{GP}, respectively. Let \(X\) be a normal projective variety and $\phi: H \rightarrow G$ be a homomorphism between reductive algebraic groups \(H\) and \(G\). Let $p : \mathcal{P}_H \rightarrow X$ be a  principal \(H\)-bundle. The associated principal \(G\)-bundle \(\mathcal{P}_G := \mathcal{P}_H \times^H  G\) is constructed as the quotient \[\left( \mathcal{P}_H \times  G \right)/ \sim \] where \[(e,g) \sim (eh,\phi(h)^{-1}g) \text{ for all } h \in H.\] The projection map \[p': \mathcal{P}_G  \rightarrow X \text{ is given by }  p'([e,g])=p(e).\] 
	We have an \(H\)-equivariant morphism 
 \begin{equation*}
     \hat{\phi} : \mathcal{P}_H \rightarrow \mathcal{P}_G, \text{ given by } \hat{\phi}(e)=[e, 1_G],
 \end{equation*}
which satisfies the following commutative diagram
\begin{equation}\label{esg1}
\begin{tikzpicture}
\def\a{1.5} \def\b{2}
\path
(-\a,0) node (A) {$\mathcal{P}_H$}      
(\a,0) node (B) {$\mathcal{P}_G$}
(0,-\b) node[align=center] (C) {$X$};
\begin{scope}[nodes={midway,scale=.75}]
\draw[->] (A)--(B) node[above]{$\hat{\phi}$};
\draw[->] (A)--(C) node[left]{$p$};
\draw[->] (B)--(C) node[right]{$p'$};
\end{scope}
\end{tikzpicture} 
 .\end{equation}
 The following proposition is an analogous version of \cite[Proposition 3.3.1]{GP} in the set-up of singular varieties. 

\begin{prop}\label{inducedconn4.1}
With the notation as above, a logarithmic connection on the principal \(H\)-bundle $\mathcal{P}_H$ induces a logarithmic connection on the associated principal $G$-bundle $\mathcal{P}_G$.
\end{prop}

\begin{proof}
Considering the maps between the tangent sheaves from \eqref{esg1}, we get the following commutative diagram
\begin{center}
    
\begin{tikzpicture}
\def\a{1.5} \def\b{2}
\path
(-\a,0) node (A) {$\mathscr{T}_{\mathcal{P}_H}$}      
(\a,0) node (B) {$ \hat{\phi}^* \mathscr{T}_{\mathcal{P}_G}$}
(0,-\b) node[align=center] (C) {$p^* \mathscr{T}_X= \hat{\phi}^* p'^* \mathscr{T}_X$};
\begin{scope}[nodes={midway,scale=.75}]
\draw[->] (A)--(B) node[above]{\(d \hat{\phi}\)};
\draw[->] (A)--(C) node[left]{\(d p\)};
\draw[->] (B)--(C) node[right]{\(\hat{\phi}^*(d p')\)};
\end{scope}
\end{tikzpicture},
\end{center}
where \(d p': \mathscr{T}_{\mathcal{P}_G} \rightarrow p'^* \mathscr{T}_X \). Since the action on $\mathcal{P}_H$ of the group \(H\) is free, the \(H\)-equivariant sheaves of the above diagram descend to sheaves on \(X\) (see \cite[Theorem 1.2]{Nevins}). Thus taking invariant pushforward \(p_*^H\) and noting that \(\mathcal{A}t(\mathcal{P}_H):=(p_*\mathscr{T}_{\mathcal{P}_H})^H\) and \((p_* p^* \mathscr{T}_{X})^H \cong \mathscr{T}_{X}\), we get the following commutative diagram

\begin{center}
    
\begin{tikzpicture}
\def\a{1.5} \def\b{2}
\path
(-\a,0) node (A) {$\mathcal{A}t(\mathcal{P}_H)$}      
(\a,0) node (B) {$(p_* \hat{\phi}^* \mathscr{T}_{\mathcal{P}_G})^H$}
(0,-\b) node[align=center] (C) {$\mathscr{T}_{X}$};
\begin{scope}[nodes={midway,scale=.75}]
\draw[->] (A)--(B) node[above]{$\text{At}(\phi)$};
\draw[->] (A)--(C) node[left]{$\eta$};
\draw[->] (B)--(C) node[right]{\(\zeta\)};
\end{scope}
\end{tikzpicture}. 
\end{center}
Now we show that \(\mathcal{A}t(\mathcal{P}_G)=(p_* \hat{\phi}^* \mathscr{T}_{\mathcal{P}_G})^H\) and \(\zeta = \eta'\), where \(\eta' : \mathcal{A}t(\mathcal{P}_G) \rightarrow \mathscr{T}_{X}\) is induced from the map \(d p'\) (see \eqref{surj}). Note that
\begin{equation*}
    \begin{split}
        & p^* \mathcal{A}t(\mathcal{P}_G)= \hat{\phi} ^* p'^* \mathcal{A}t(\mathcal{P}_G)= \hat{\phi} ^* \mathscr{T}_{\mathcal{P}_G} \text{ and }\\
        & p^*((p_* \hat{\phi}^* \mathscr{T}_{\mathcal{P}_G})^H)= \hat{\phi} ^* \mathscr{T}_{\mathcal{P}_G} ~ (\text{since } \hat{\phi} ^* \mathscr{T}_{\mathcal{P}_G} \text{ descends to } (p_* \hat{\phi}^* \mathscr{T}_{\mathcal{P}_G})^H).
    \end{split}
\end{equation*}
Thus both the sheaves \(\mathcal{A}t(\mathcal{P}_G)\) and \((p_* \hat{\phi}^* \mathscr{T}_{\mathcal{P}_G})^H\) are descends of the same sheaf. Hence, we have \(\mathcal{A}t(\mathcal{P}_G)=(p_* \hat{\phi}^* \mathscr{T}_{\mathcal{P}_G})^H\). Also, we have
\begin{equation*}
    \begin{split}
        p^*(\zeta) &= \hat{\phi}^*(d p') ~ (\text{since \(\hat{\phi}^*(d p')\) descends to the map \(\zeta\)})\\
        &=\hat{\phi}^* (p'^*(\eta') ~ (\text{since \(d p'\) descends to the map \(\eta\)})\\
        &= p^* (\eta') ~ (\text{using } p' \circ \hat{\phi}=p).
    \end{split}
\end{equation*}
Now using the fact that the functor \(p_*^H \, p^*\) is the identity functor (see \cite[Subsection 2.2]{Nevins}), we get \(\zeta=\eta'\). Thus, we get the following commutative diagram of the corresponding Atiyah sequences.

    	\begin{equation*}
	\xymatrixrowsep{1.8pc} \xymatrixcolsep{2.2pc}
	\xymatrix{
		0\ar@{->}[r]&
		\text{ad}(\mathcal{P}_H)\ar@{->}[r]^{\iota}\ar@{->}[d]^{\text{ad}(\phi)}& \mathcal{A}t(\mathcal{P}_H)\ar@{->}[d]^{\text{At}(\phi)}\ar@{->}[r]^{\eta}&\mathscr{T}_X\ar@{->}[d]_{Id}\ar@{->}[d]\ar@{->}[r]&0\\
		0\ar@{->}[r]&
		\text{ad}(\mathcal{P}_G)\ar@{->}[r]^{\iota'}&\mathcal{A}t(\mathcal{P}_G)\ar@{->}[r]^{\eta'}&\mathscr{T}_X\ar@{->}[r]&{0}.
	}
			\end{equation*}
Here $\text{ad}(\phi)$ is the induced map between \(\text{Ker}(\eta)\) and \(\text{Ker}(\eta')\) from the map \(\text{At}(\phi)\). This induces the following commutative diagram with exact rows
   \begin{equation}\label{esg1.1}
	\xymatrixrowsep{1.8pc} \xymatrixcolsep{2.2pc}
	\xymatrix{
		0\ar@{->}[r]&
		\text{ad}(\mathcal{P}_H)\ar@{->}[r]^{\iota}\ar@{->}[d]^{\text{ad}(\phi)}& \mathcal{A}t(\mathcal{P}_H)\ar@{->}[d]^{\text{At}(\phi)}\ar@{->}[r]^{\bar{\eta}}&\mathscr{T}_X(-\text{log } D)\ar@{->}[d]_{Id}\ar@{->}[d]\ar@{->}[r]&0\\
		0\ar@{->}[r]&
		\text{ad}(\mathcal{P}_G)\ar@{->}[r]^{\iota'}&\mathcal{A}t(\mathcal{P}_G)\ar@{->}[r]^{\bar{\eta'}}&\mathscr{T}_X(-\text{log } D)\ar@{->}[r]&{0}.
	}
			\end{equation}
   If the principal $H$-bundle $\mathcal{P}_H$ admits a logarithmic connection $\zeta$, then from the commutativity of \eqref{esg1.1}, $\text{At}(\phi) \circ \zeta$ will be a logarithmic connection on $\mathcal{P}_G$.
\end{proof}

\begin{cor}
    Let $p : \mathcal{P} \rightarrow X$ be a  principal \(G\)-bundle with a logarithmic connection. Then for any representation \(V\) of \(G\), there is an induced logarithmic connection on the associated vector bundle \(\mathcal{P}(V):=\mathcal{P} \times^G  V\)
\end{cor}

\noindent
The following proposition shows that the converse of Proposition \ref{inducedconn4.1} holds under the assumption that the map $\phi$ is injective (cf. \cite[Proposition 3.3.1]{GP} and \cite[Lemma 3.3]{BDPS}).

\begin{prop}
    Let $\phi : H \rightarrow G$ be an injective homomorphism between reductive algebraic groups \(H\) and \(G\). Let $p : \mathcal{P}_H \rightarrow X$ be a  principal \(H\)-bundle. Then a logarithmic connection on the associated principal \(G\)-bundle $\mathcal{P}_G:= \mathcal{P}_H \times^H  G$ induces a logarithmic connection on the principal $H$-bundle $\mathcal{P}_H$.
\end{prop} 

\begin{proof}
    The map $\phi$ induces an injective homomorphism \(d\phi : \mathfrak{h} \rightarrow \mathfrak{g}\) between the corresponding Lie algebras. Note that \(\mathfrak{h}\) is an \(H\)-module via the adjoint action. Also, the adjoint action of \(G\) composed with $\phi$ makes $\mathfrak{g}$ an \(H\)-module. The map \(d\phi : \mathfrak{h} \rightarrow \mathfrak{g}\) preserves the \(H\)-module structures. Since \(H\) is reductive, there is an \(H\)-submodule \(V\), such that \[\mathfrak{g}= d\phi(\mathfrak{h}) \oplus V.\]
    Since \(d\phi\) is injective, identifying \(\mathfrak{h}\) with its image, we have the projection map \(\gamma : \mathfrak{g} \rightarrow \mathfrak{h}\). In particular, we have 
    \begin{equation}\label{4.9eq1}
        \gamma \circ d\phi =Id_{\mathfrak{h}}.
    \end{equation} 
 Since \(\text{ad}(\mathcal{P}_G) \cong \mathcal{P}_H \times^H \mathfrak{g}\), the map $\gamma$ induces the following map 
    $$\bar \gamma : \text{ad}(\mathcal{P}_G) \cong \mathcal{P}_H \times^H \mathfrak{g}   \longrightarrow   \text{ad}(\mathcal{P}_H):= \mathcal{P}_H \times^H \mathfrak{h}.$$
Note that the map \(\text{ad}(\phi)\) in \eqref{esg1.1} is induced from the map 
\begin{equation}\label{4.9eq2}
    Id \times d \phi : \mathcal{P}_H \times \mathfrak{h} \rightarrow \mathcal{P}_H \times \mathfrak{g}.
\end{equation}
To see this first observe that \(\text{ad}(\phi)\) is induced from \(Id \times d \phi\) when the bundle $\mathcal{P}_H$ is trivial. Since the bundle \(\mathcal{P}_H\) is locally isotrivial, we pull back the principal bundles in the commutative diagram \eqref{esg1} via an étale trivialization to reduce to the case of the trivial bundle. Then we argue as in the proof of Proposition \ref{atseq}. Hence from \eqref{4.9eq1} and \eqref{4.9eq2}, we have
    \[\bar \gamma \circ \text{ad}(\phi)= Id_{\text{ad}(\mathcal{P}_H)}.\]
    Now suppose that $\mathcal{P}_G$ admits a logarithmic connection $\zeta'$. Then it induces a splitting 
    \[\zeta'' : \mathcal{A}t(\mathcal{P}_G) \rightarrow \text{ad}(\mathcal{P}_G) \] 
    of $\iota'$ in \eqref{esg1.1}. Then \(\bar \gamma \circ \zeta'' \circ \text{At}(\phi)\) gives a splitting of $\iota$. Hence, we get a splitting of $\bar \eta$, i.e. a logarithmic connection on \(\mathcal{P}_H\).
\end{proof}

\begin{cor}\label{adbdlstab}
    Let \(G\) be a semisimple linear algebraic group and $\mathcal{P}$ be a principal \(G\)-bundle on \(X\). Then $\mathcal{P}$ admits a logarithmic connection if and only if  the associated adjoint bundle \(\text{ad}(\mathcal{P})\) admits a logarithmic connection.
\end{cor}
 
	\section{Logarithmic connections on equivariant principal bundles over toric varieties}\label{toric situation}
	
	In this section, we show that any toric principal bundle over a toric variety admits a natural logarithmic connection. We extend the notion of residue for any vector bundle on toric varieties. Finally, we establish a relation between the residue of the natural connection of an equivariant vector bundle with its equivariant structure.
	
	\subsection{Existence of logarithmic connections on toric varieties}\label{log conn toric}
	
	Let \(X\) be a projective toric variety. A toric principal \(G\)-bundle on the toric variety \(X\) is a principal \(G\)-bundle \(p : \mc{P} \rightarrow X\) together with a lift of the \(T\)-action on the total space $\mc{P}$ which commutes with the right \(G\)-action.
	
	The following proposition shows that on a toric variety \(X\), any toric principal bundles admit a logarithmic connection.
	\begin{prop}\label{prop1}
		Suppose \(\mathcal{P}\) is a $T$-equivariant algebraic principal $G$-bundle over a toric variety $X$ with boundary divisor \(D=X \setminus T\). Then \(\mathcal{P}\) admits an integrable logarithmic connection singular over \(D\).
	\end{prop}
	
	\begin{proof}
		Consider the sheaf \(\tilde{\mathcal{V}}:= \mathcal{O}_{\mathcal{P}} \otimes_{\Z} N \cong  \mathcal{O}_{\mathcal{P}} \otimes_{\mathbb{K}} \mathfrak{t}  \) on \(\mathcal{P}\).  Note that, \(\text{H}^0(\mathcal{P}, \tilde{\mathcal{V}})\) has a natural Lie algebra structure via the isomorphism \(\text{H}^0(\mathcal{P}, \tilde{\mathcal{V}}) \cong \mathfrak{t}\). Since there is an action of \(T\) on \(\mathcal{P}\), as before we have \( \mathfrak{t} \subset \text{H}^0(\mathcal{P}, \mathscr{T}_{\mathcal{P}})\). To see this, note that for any \(T\)-invariant open set \(V \subset \mathcal{P} \), we have a \(T\)-isotypical decomposition \(\mathcal{O}_{\mathcal{P}}(V)= \oplus_{m \in M} \, \mathcal{O}_{\mathcal{P}}(V)_m\). Then for \(v \in N\), define the derivation (cf. \eqref{LID on X})
		\begin{equation}\label{del_tilde}
			\tilde{\delta}_v:\mathcal{O}_{\mathcal{P}}(V) \rar \mathcal{O}_{\mathcal{P}}(V) \text{ by sending } \tilde{f} \mapsto \langle m,v \rangle \, \tilde{f},
		\end{equation}
		where \(\tilde{f} \in \mathcal{O}_{\mathcal{P}}(V)_m\). Since \(X\) is normal, the principal bundle $\mathcal{P}$ is also normal (by \cite[Section 3]{Br}) and hence \(\mathcal{P}\) has a covering by \(T\)-invariant affine open subsets using Sumihiro's theorem \cite[Theorem 3.1.7]{Cox}. Thus the derivations \(\tilde{\delta}_v\) defined in \eqref{del_tilde} glue to give an element of \(\text{H}^0(\mathcal{P}, \mathscr{T}_{\mathcal{P}})\). Thus, we have a morphism
		\begin{equation}\label{trivial_top}
			\tilde{\beta}: \tilde{\mathcal{V}} \rar \mathscr{T}_{\mathcal{P}}.
		\end{equation} 
		Note that $\tilde{\mathcal{V}}$ can be given a Lie algebra structure similarly as before (see \eqref{Lie_str}) so that \(\tilde{\beta}\) becomes a Lie algebra homomorphism. We now show that \(\tilde{\beta}\) is \(G\)-equivariant. First note that for \(v \in N\), the derivation $\tilde{\delta_v}$ defined in \eqref{del_tilde} is a \(G\)-invariant derivation. To see this, observe that there is a cover of $\mathcal{P}$ by open sets \(\{V\}\), which are both \(G\) and \(T\)-invariant, namely of the form \(\{p^{-1}(U_i)\}\), where \(\{U_i\}\) is a cover of \(X\).  Let $g \in G$ and \(\tilde{f} \in \mathcal{O}_{\mathcal{P}}(V)_m\), we have
		\begin{equation}\label{tildel_inv}
			\begin{split}
				(\tilde{\delta_v} \cdot g)(\tilde{f}) &=\tilde{\delta_v}(\tilde{f} \cdot g^{-1}) \cdot g\\
				&=(\langle m,v \rangle \, \tilde{f} \cdot g^{-1}) \cdot g ~ (\text{using \eqref{del_tilde} and the fact that }\\
				& \hspace{4 cm} \tilde{f} \cdot g^{-1} \in \mathcal{O}_{\mathcal{P}}(V)_m \, \text{as } T \text{ and } G \text{ actions commute}) \\
				&=\tilde{\delta_v} (\tilde{f}) ~ (\text{again by \eqref{del_tilde}}).
			\end{split}
		\end{equation}
		So for \(v \in N\) and \(g \in G\), we have
		\begin{equation*}
			\begin{split}
				(\tilde{\beta}_V(\tilde{f}_1 \otimes v) \cdot g)(\tilde{f}) &=((\tilde{f}_1 \tilde{\delta_v}) \cdot g)(\tilde{f})\\
				&=(\tilde{f}_1 \tilde{\delta_v}(\tilde{f} \cdot g^{-1}) )\cdot g\\
				&=(\tilde{f}_1 \cdot g) (\tilde{\delta_v}(\tilde{f} \cdot g^{-1}) \cdot g)\\
				& =(\tilde{f}_1 \cdot g) (\tilde{\delta_v} \cdot g)(\tilde{f})\\
				& =(\tilde{f}_1 \cdot g) (\tilde{\delta_v} )(\tilde{f}) ~ (\text{by } \eqref{tildel_inv})\\
				&=\tilde{\beta}_V((\tilde{f}_1 \cdot g \otimes v) )(\tilde{f})\\
				&=\tilde{\beta}_V(((\tilde{f}_1  \otimes v) \cdot g) )(\tilde{f}).
			\end{split}	
		\end{equation*}
		
		\noindent
		Thus, taking the invariant direct image of \eqref{trivial_top}, i.e., first considering the map between push forwards 
		\begin{equation*}
			p_*\tilde{\beta}: p_*\tilde{\mathcal{V}}=(p_*\mathcal{O}_{\mathcal{P}}) \otimes_{\Z} N \rar p_*\mathscr{T}_{\mathcal{P}}
		\end{equation*}
		and then taking the subsheaves of \(G\)-invariants on both sides and noting that the natural map in \eqref{p-sharp} induces  \((p_*\mc{O_{\mathcal{P}}})^G \cong \mc{O}_X\), we obtain the \(\mathcal{O}_X\)-module morphism
		\begin{equation}\label{gamma}
			(p_*\tilde{\beta})^G : \mathcal{V} \rar \mathcal{A}t(\mathcal{P}).
		\end{equation}
		We see that this morphism also preserves Lie algebra structures since \(\tilde{\beta}\) is \(G\)-equivariant and preserves Lie algebra structures. Note that for any \(v \in N\) we have,
		\begin{equation*}
			((p_*\tilde{\beta})^G)(1 \otimes v)=\tilde{\delta_v}.
		\end{equation*}
		Let \(f \in \mathcal{O}_X(U)\) with weight \(m\), i.e. \(t f= \chi^m(t) f\) for all \(t \in T\). Then 
		\begin{equation*}
			\begin{split}
				\eta(\tilde{\delta}_v)(f)= & p_* (\tilde{\delta}_v(p^{\sharp}(f))) ~ (\text{by \eqref{surj}})\\
				= & p_* (\langle m, v \rangle p^{\sharp}(f)) ~ (\text{since } p \text{ is \(T\)-equivariant, } p^{\sharp}(f) \text{ also has weight } m)\\
				= & \langle m, v \rangle f  ~ (\text{since } p_*p^{\sharp}(f)=f)\\
				=&\delta_v(f).
			\end{split}
		\end{equation*}
		Hence we have \(\eta(\tilde{\delta}_v)=\delta_v\).	Thus we have the following commutative diagram of \(\mathcal{O}_X\)-modules
		\begin{center}
			\begin{tikzpicture}
				\matrix (m) [matrix of math nodes,row sep=3em,column sep=4em,minimum width=2em] {
					\mathcal{V} & \mathcal{A}t(\mathcal{P}) \\
					\mathcal{V} & \mathscr{T}_X \\};
				\path[-stealth]
				(m-1-1) edge node [left] {Id} (m-2-1)
				edge  node [above] {\((p_*\tilde{\beta})^G\)} (m-1-2)
				(m-2-1) edge node [below] {$\beta$} (m-2-2)
				(m-1-2) edge node [right] {$\eta$} (m-2-2);
			\end{tikzpicture}	
		\end{center}
		From \eqref{trivial2}, it follows that \eqref{gamma} factors through a map
		\begin{equation*}
			\gamma : \mathcal{V} \rar \mathcal{A}t(\mathcal{P})(- \text{log }D).
		\end{equation*} Since the map in \eqref{gamma} preserves Lie algebra structures, it follows that \(\gamma\) also preserves Lie algebra structures. 
		Consider \(\zeta:=\gamma \circ \bar{\beta}^{-1}\), the composition of \(\gamma\) with the isomorphism \(\bar{\beta}^{-1}\), given by
		\begin{equation*}
			\zeta: \mathscr{T}_X(-\text{log }D) \rar \mathcal{A}t(\mathcal{P})(- \text{log }D),
		\end{equation*}
		(see \eqref{trivial2}).
		Then we have \[\bar{\eta} \circ (\gamma \circ \bar{\beta}^{-1})=\text{Id}_{\mathscr{T}_X(-\text{log }D)}.\]
		Hence \(\zeta\) defines a logarithmic connection on the principal \(G\)-bundle $\mathcal{P}$ with singular locus \(D\). Since the maps $\bar{\beta}$ and $\gamma$ preserve Lie algebra structures, the logarithmic connection \(\zeta\) is integrable.
	\end{proof}
	
	\begin{rmk}\label{image}
		{\rm With notations as above, observe that \(\zeta(\delta_v)=\tilde{\delta}_v\).}
	\end{rmk}
	
	\begin{rmk}\label{vanishing}
		{\rm Let \(E\) be a \(T\)-equivariant vector bundle on a toric variety \(X\). Then the associated principal \(GL(r, \mathbb{K})\)-bundle $\mathcal{P}$
			is also equivariant. So $\mathcal{P}$ admits a logarithmic connection $\zeta$ using Proposition \ref{prop1}. Let $\nabla$ be the connection on the vector bundle \(E\) induced using Proposition \ref{equivalence1}. Let \(s\) be a \(T\)-invariant section of \(E\) over an invariant open subset of \(X\). Since $\Phi$ is also \(T\)-equivariant, the regular functions $\tilde{f}_1, \ldots, \tilde{f}_r$ defined in \eqref{img under Phi} are also \(T\)-invariant. Hence
			\begin{equation*}
				\begin{split}
					\zeta(\delta_{v_i})(\tilde{f}_i) & =\tilde{\delta}_{v_i}(\tilde{f}_i) ~ (\text{see Remark \ref{image}})\\
					& =0 ~ (\text{since } \tilde{f}_i \text{ has weight }0). 
				\end{split}
			\end{equation*}
			Hence from \eqref{nabla1} we see that $\widetilde{\nabla}(s \otimes \delta_{v_i})=0$, which implies that $\nabla(s)=0$ from \eqref{nabla2}. In other words, the connection vanishes on invariant sections of \(E\).
		} 
	\end{rmk}
	
	\begin{rmk}
		{\rm We also have an alternative proof of Proposition \ref{prop1}, which we give below. Let \(U_0=\cup_{\rho \in \Sigma(1)} U_{\rho}\) be the smooth toric variety whose fan has the same 1-dimensional cones as $\Sigma$, \(\text{codim}(X \setminus U_0) \geq 2 \) by the orbit-cone correspondence. Let \[D_0=\sum_{\rho \in \Sigma(1)} (D_{\rho} \cap U_0) \] denote the restriction of the boundary divisor to \(U_0\). Consider a toric principal \(G\)-bundle 
			\begin{equation*}
				\begin{split}
					&p: \mathcal{P} \rightarrow X, \text{ denote by}\\ &\mathcal{P}_0:=\mathcal{P}|_{U_0} \text{ and } \mathscr{T}_{0}:=\mathscr{T}_X|_{U_0}.
				\end{split}
			\end{equation*}
			Since both \(U_0 \subset X\) and \(\mathcal{P}_0 \subset \mathcal{P}\) are open subsets, we have
			\begin{equation*}
				\begin{split}
					&\text{ad}(\mathcal{P}_0)=\text{ad}(\mathcal{P})|_{U_0},\\ &\mathcal{A}t(\mathcal{P}_0) (- \text{log } D_0)=\mathcal{A}t(\mathcal{P}) (- \text{log } D)|_{U_0} \text{ and }\\
					& \mathscr{T}_{0}(- \text{log } D_0)=\mathscr{T}_X(- \text{log } D)|_{U_0}.
				\end{split}
			\end{equation*}
			Then restricting the logarithmic Atiyah sequence \eqref{log_at} to the smooth locus \(U_0\), we have the following exact sequence:
			\begin{equation}\label{log_at_0}
				0 \longrightarrow \text{ad}(\mathcal{P}_0) \longrightarrow \mathcal{A}t(\mathcal{P}_0) (- \text{log } D_0)  \stackrel{\bar{\eta}_0} \longrightarrow  \mathscr{T}_{0}(- \text{log } D_0) \longrightarrow 0 .
			\end{equation}
			Let $$\zeta_0:  \mathscr{T}_0(- \text{log } D_0) \rightarrow \mathcal{A}t(\mathcal{P}_0) (- \text{log } D_0)$$ be the integrable logarithmic connection given by \cite[Proposition 3.2]{BDP}, in other words,  \[\bar{\eta}_0 \circ \zeta_0 = Id_{\mathscr{T}_0(- \text{log } D_0)}.\] Note that $\zeta_0$ corresponds to a section \[s_0 \in \Gamma(U_0, \mathscr{T}_X(- \text{log } D)^{\vee} \otimes \mathcal{A}t(\mathcal{P}) (- \text{log } D)).\] Since \(\mathscr{T}_X(- \text{log } D)^{\vee} \otimes \mathcal{A}t(\mathcal{P}) (- \text{log } D)\) is locally free and \(\text{codim}(X \setminus U_0) \geq 2 \), the section \(s_0\) extends uniquely to a global section of \(\mathscr{T}_X(- \text{log } D)^{\vee} \otimes \mathcal{A}t(\mathcal{P}) (- \text{log } D)\). This global section induces a map  $$\zeta:  \mathscr{T}_X(- \text{log } D) \rightarrow \mathcal{A}t(\mathcal{P}) (- \text{log } D)$$ which also satisfies \[\bar{\eta} \circ \zeta = Id_{\mathscr{T}_X(- \text{log } D)}\] by density of \(U_0\). Note that this map is a unique extension of \(\zeta_0\). But \(\zeta_0\) is integrable, then for any \(\delta_1, \delta_2\), local sections of \(\mathscr{T}_X(- \text{log } D)\), we have \[\zeta_0([\delta_1 \mid_{U_0}, \delta_2 \mid_{U_0}])=[\zeta_0(\delta_1 \mid_{U_0}), \zeta_0(\delta_2 \mid_{U_0})].\] Then by uniqueness of the extension \(\zeta\) and the injectivity of restriction maps of  locally free sheaves, we see that \[\zeta([\delta_1, \delta_2])=[\zeta(\delta_1), \zeta(\delta_2)].\]  Hence the principal bundle \(p: \mathcal{P} \rightarrow X\) admits an integrable logarithmic connection.
		}
	\end{rmk}

	\subsubsection{Equivariant logarithmic connection}
	
	Let $\mathcal{P}$ be a \(T\)-equivariant principal \(G\)-bundle on the projective toric variety \(X\). Note that the Atiyah sheaf \(\mathcal{A}t(\mathcal{P})\) can be given a \(T\)-equivariant structure as follows. Let \(U\) be a \(T\)-invariant open subset of \(X\). Then \(p^{-1}(U)\) is a \(T\)-invariant open subset of $\mathcal{P}$, as the projection map \(p\) is \(T\)-equivariant. Then the action of \(T\) on \(\mathcal{A}t(\mathcal{P})(U)=\text{Der}_{\mathbb{K}}(\mathcal{O}_{\mc P}(p^{-1}(U)), \mathcal{O}_{\mc P}(p^{-1}(U)))^G\) is given by
	\begin{equation*}
		(t \cdot \tilde{\delta})(\tilde{f})=t \cdot \tilde{\delta}(t^{-1} \cdot \tilde{f}),
	\end{equation*}
	for \(t \in T\), $\tilde{\delta} \in \mathcal{A}t(\mathcal{P})(U)$ and \(\tilde{f} \in \mathcal{O}_{\mc P}(p^{-1}(U))\). The tangent sheaf $\mathscr{T}_X$ has a natural \(T\)-equivariant structure so that the map $\eta: \mathcal{A}t(\mathcal{P}) \rightarrow \mathscr{T}_X$ defined in \eqref{the map eta} becomes \(T\)-equivariant. Moreover, the natural inclusion \(\mathscr{T}_X(-\text{log }D) \hookrightarrow \mathscr{T}_X\) is also \(T\)-equivariant, hence \(\mathcal{A}t(\mathcal{P})(-\text{log }D)\) acquires a \(T\)-equivariant structure so that the logarithmic Atiyah sequence \eqref{log_at} becomes a short exact sequence of \(T\)-equivariant vector bundles. In particular, $\bar{\eta}$ becomes \(T\)-equivariant.
	
	Set \(\mathcal{A}:=\mathcal{A}t(\mathcal{P})(-\text{log }D)\) and \(\mathscr{T}:=\mathscr{T}_X(-\text{log }D)\). There is an induced action of \(T\) on \(\text{Hom}_{\mathcal{O}_X}(\mathscr{T}, \mathcal{A})\), the space of $\mathcal{O}_X$-module homomorphism between the vector bundles \(\mathscr{T}\) and \(\mathcal{A}\). Hence we have the following isotypical decomposition
	\begin{equation*}
		\text{Hom}_{\mathcal{O}_X}(\mathscr{T}, \mathcal{A})= \bigoplus\limits_{m \in M} \, \text{Hom}_T(\mathscr{T}, \mathcal{A} \otimes \text{div}(\chi^m)), 
	\end{equation*}
	where \(\text{Hom}_T(\mathscr{T}, \mathcal{A} \otimes \text{div}(\chi^m))\) denotes the space of \(T\)-equivariant homomorphisms between the vector bundles \(\mathscr{T}\) and \(\mathcal{A} \otimes \text{div}(\chi^m)\).
	
	\begin{defn}
		A homogeneous logarithmic connection of degree \(m\in M\) is an element $\zeta_m \in \text{Hom}_T(\mathscr{T}, \mathcal{A} \otimes \text{div}(\chi^m))$ such that \(\bar{\eta} \circ \zeta_m = Id_{\mathscr{T}_X(- \text{log } D)}\) holds. A homogeneous logarithmic connection of degree \(0\) is said to be an equivariant logarithmic connection.
	\end{defn}
	
	\begin{rmk}{\rm
			Let $\zeta$ be a logarithmic connection on the equivariant principal \(G\)-bundle $\mc{P}$. Then $\zeta \in 	\text{Hom}_{\mathcal{O}_X}(\mathscr{T}, \mathcal{A})$. Then write \(\zeta=\sum_{m \in M} \zeta_m\), where \(\zeta_m \in 	\text{Hom}_{\mathcal{O}_X}(\mathscr{T}, \mathcal{A})_m\). We have
			\begin{equation}\label{eqi_conn1}
				Id_{\mathscr{T}}=\bar{\eta} \circ \zeta= \sum_{m \in M} (\bar{\eta} \circ \zeta_m).
			\end{equation}
			Note that \(\bar{\eta} \circ \zeta_m \in \text{Hom}_{\mathcal{O}_X}(\mathscr{T}, \mathscr{T})\) and  \(Id_{\mathscr{T}} \in \text{Hom}_{\mathcal{O}_X}(\mathscr{T}, \mathscr{T})\), hence comparing degree on both sides of \eqref{eqi_conn1}, we have 
			\begin{equation*}
				\bar{\eta} \circ \zeta_0=Id_{\mathscr{T}}  \text{ and } \bar{\eta} \circ \zeta_m= 0 \text{ for } 0 \neq m \in M.
			\end{equation*}
			This shows that $\zeta_0$ is an equivariant logarithmic connection. 
		}
	\end{rmk}
	
	\begin{rmk}{\rm
			Let $\zeta$ be the integrable logarithmic connection on the equivariant principal \(G\)-bundle $\mc{P}$ given by the Proposition \ref{prop1}. We note that $\zeta$ itself is an equivariant logarithmic connection. This can be seen as follows. Since the projection \(p : \mc{P} \rightarrow X\) is \(T\)-equivariant, one can show that the map \(\tilde{\beta}\) defined in \eqref{trivial_top}, is \(T\)-equivariant. The induced map \((p_*\tilde{\beta})^G\) defined in \eqref{gamma}, is again \(T\)-equivariant, since the actions of \(G\) and \(T\) commute. This shows that $\zeta$ being a composition of \(T\)-equivariant maps is also \(T\)-equivariant.
		}
	\end{rmk}

	\subsection{Equivariant structure induced from a logarithmic connection}\label{eqslc}
	
Let \(X\) be a projective toric variety. To any principal \(G\)-bundle $p : \mathcal{P} \rightarrow X $ we associate the following (abstract) groups (see \cite[Section 4]{Br}):
	\begin{align*}
		& \text{Aut}_X(\mathcal{P}):=\{\tilde{\psi} \in  \text{Aut}(\mathcal{P})  ~|~  p \circ \tilde{\psi}=p\},\\
		& \text{Aut}(\mathcal{P}, X):=\{(\tilde{\psi}, \psi)  \in \text{Aut}(\mathcal{P}) \times \text{Aut}(X) ~|~ p \circ \tilde{\psi}=\psi \circ p\},\\
		& \text{Aut}^G(\mathcal{P}):=\{\tilde{\psi} \in \text{Aut}(\mathcal{P})  ~|~ \tilde{\psi} \text{ is $G$-equivariant}\},\\
		& \text{Aut}^G_X(\mathcal{P}):=\{\tilde{\psi} \in  \text{Aut}(\mathcal{P})  ~|~ \tilde{\psi} \text{ is $G$-equivariant, and } p \circ \tilde{\psi}=p\}.
	\end{align*} 
	We have an exact sequence 
	\begin{equation*}
		1 \longrightarrow \text{Aut}_X(\mathcal{P}) \longrightarrow \text{Aut}(\mathcal{P}, X) \stackrel{p_2}\longrightarrow \text{Aut}(X) .
	\end{equation*}
	Since \(p : \mathcal{P} \rightarrow X\) is a categorical quotient, for each \(G\)-equivariant automorphism \(\tilde{\psi}:\mathcal{P} \rar \, \mathcal{P}\), we have the \(G\)-equivariant map \(p \circ \tilde{\psi}: \mathcal{P} \rar X\), which descends to an automorphism of \(X\), say \(h(\tilde{\psi}): X \rar \, X\) satisfying \(h(\tilde{\psi}) \circ p=p \circ \tilde{\psi}\).  This induces a map 
	\begin{equation}\label{map_aut} 
		h : \text{Aut}^G(\mathcal{P}) \rightarrow \text{Aut}(X).
	\end{equation}
This gives rise to an identification of \(\text{Aut}^G(\mathcal{P})\) with a subgroup of \(\text{Aut}(\mathcal{P}, X)\)
	and we have the following exact sequence
	\begin{equation}\label{autfn1}
		1 \longrightarrow \text{Aut}^G_X(\mathcal{P}) \longrightarrow \text{Aut}^G(\mathcal{P}) \stackrel{h}\longrightarrow \text{Aut}(X).
	\end{equation}
	
	Now recall from \cite[Section 3]{MO} that given any scheme \(V\) over any field $\mathbb{K}$ of characteristic zero, we can associate the automorphism group functor 
	\begin{equation*}
Aut (V) : (Sch/ \mathbb{K})^{op} \rightarrow Gr \text{ given by } Aut  (V)(S)=\text{Aut}_S(V \times S),
	\end{equation*}
	where \(S \in Sch/ \mathbb{K}\). If this functor is representable, we say that the automorphism group scheme of \(V\) exists. With this notion, the exact sequence \eqref{autfn1} yields the following exact sequence of group functors
	
	\begin{equation}\label{autfn2}
		1 \longrightarrow Aut^G_X(\mathcal{P}) \longrightarrow Aut^G(\mathcal{P}) \stackrel{h}\longrightarrow Aut(X).
	\end{equation}
	The functor $Aut^G(\mathcal{P})$ is represented by a group scheme, locally of finite type by \cite[Theorem 4.2]{Br}, which we will denote by $\text{Aut}^G(\mathcal{P})$. In particular, the neutral component \((Aut^G(\mathcal{P}))^{\circ}\) is a group scheme of finite type, i.e. smooth group variety. Since $Aut^G_X(\mathcal{P})$ is a closed subfunctor of $Aut^G(\mathcal{P})$, it is also represented by a group scheme (locally of finite type) that we denote likewise by $\text{Aut}^G_X(\mathcal{P})$. Moreover, $(\text{Aut}^G_X(\mathcal{P}))^{\circ}$ is a linear algebraic group (see \cite[Proposition 4.3]{Br}). Also, note that \(\text{Aut}(X)\) is an affine algebraic group with the torus \(T\) being a maximal torus (\cite{BG, Bu}). Thus, we obtain the following exact sequence of algebraic groups
	
		\begin{equation*}
		1 \longrightarrow (\text{Aut}^G_X(\mathcal{P}))^{\circ} \longrightarrow (\text{Aut}^G(\mathcal{P}))^{\circ} \stackrel{h}\longrightarrow (\text{Aut}(X))^{\circ}.
	\end{equation*}
	Consider the following exact sequence
	
	\begin{equation}\label{autfn3}
		1 \longrightarrow (\text{Aut}^G_X(\mathcal{P}))^{\circ} \longrightarrow (\text{Aut}^G(\mathcal{P}))^{\circ} \stackrel{h}\longrightarrow \text{Im}(h) \longrightarrow 1.
	\end{equation}
Note that \(\text{Im}(h)\) is linear algebraic group,	since \(\text{Im}(h)\) is a closed subgroup of \((\text{Aut}(X))^{\circ}\). Then there is a linear algebraic group structure on \((\text{Aut}^G(\mathcal{P}))^{\circ}\) uniquely determined by the exact sequence \eqref{autfn3}. For simplicity of notation, we omit to denote $\circ$.
	
	We now recall the Lie algebra of these algebraic groups from \cite{MO}. Consider the ring of dual numbers \(\text{I}_{\mathbb{K}}=\frac{\mathbb{K}[\varepsilon]}{\langle \varepsilon^2 \rangle}\) and let the augmentation map \(r:\text{I}_{\mathbb{K}} \rar \, \mathbb{K}\) be given by \(\varepsilon \mapsto 0\).
	We have \begin{equation*}
		\begin{split}
			\text{Lie}(\text{Aut}(X))=&\{\varphi:\text{Spec }\text{I}_{\mathbb{K}} \rar \text{Aut}(X) ~|~ \\
			&\varphi(*)=id, \text{ where Spec }\text{I}_{\mathbb{K}}=\{*\}  \}~ (\text{see \cite[Section 2]{MO}})\\
			 =\{\phi:& \mc{O}_X \rar \text{I}_{\mathbb{K}} \otimes \mc{O}_X		 \text{ is a homomorphism of }\\ & \mathbb{K} \text{-algebras on } X ~|~ (r \otimes 1) \circ \phi= \text{id}\} ~ (\text{see \cite[Lemma 3.3]{MO}}).
		\end{split}
	\end{equation*}

	and similarly, we have
		\begin{equation*}
		\begin{split}
			\text{Lie}(\text{Aut}^G(\mathcal{P}))=&\{\varphi:\text{Spec }\text{I}_{\mathbb{K}} \rar \text{Aut}^G(\mathcal{P}) ~|~ \\
			&\varphi(*)=id, \text{ where Spec }\text{I}_{\mathbb{K}}=\{*\}  \}~ (\text{see \cite[Section 2]{MO}})\\
			=&\{\phi: \mc{O_P} \rar \, \text{I}_{\mathbb{K}} \otimes \mc{O_P}		 \text{ is a } G\text{-equivariant homomorphism of }\\ 
			& \mathbb{K} \text{-algebras on } \mc{P} ~|~ (r \otimes 1) \circ \phi= \text{id}\} ~ (\text{see \cite[Lemma 3.3]{MO}})
		\end{split}
	\end{equation*}
	The map \(h\) defined in \eqref{map_aut} gives rise to the following map between their Lie algebras
	\begin{equation*}
		dh:\text{Lie}(\text{Aut}^G(\mathcal{P})) \rar 	\text{Lie}(\text{Aut}(X)) \text{ given by } \phi \mapsto dh(\phi).
	\end{equation*}
	To see this, consider \(\phi: \mc{O_P} \rar \, \text{I}_{\mathbb{K}} \otimes \mc{O_P} \in \text{Lie}(\text{Aut}^G(\mathcal{P})) \).  This will induce a map \(p_* \phi : p_*\mc{O_P} \rightarrow	I_{\mathbb{K}} \otimes  p_*\mc{O_P} \). Recall that \((p_*\mc{O_P})^G=\mc{O}_X\), so taking \(G\)-invariant part, we get a map 
	\begin{equation*}
		dh(\phi) : \mc{O}_X  \rightarrow  I_{\mathbb{K}} \otimes \mc{O}_X,
	\end{equation*}
	which fits into the following commutative diagram:
	
	\begin{center}
		\begin{figure}[h] 
			\begin{tikzpicture}
				\matrix (m) [matrix of math nodes,row sep=3em,column sep=4em,minimum width=2em] {\mc{O}_X  & I_{\mathbb{K}} \otimes \mc{O}_X \\
					p_*\mc{O_P} & 	I_{\mathbb{K}} \otimes  p_*\mc{O_P} \\
				};
				\path[-stealth]
				(m-1-1) edge node [left] {$p^{\sharp}$} (m-2-1)
				edge  node [above] {$dh(\phi)$} (m-1-2)
				(m-2-1) edge node [below] {$p_* \phi$} (m-2-2)
				(m-1-2) edge node [right] {$p^{\sharp}$} (m-2-2);
			\end{tikzpicture}	
			\caption{}
			\label{lie}
		\end{figure}
	\end{center} 
	This implies that 
	\begin{equation}\label{key}
		\phi(p^{\sharp}(f))=(p^{\sharp}(dh(\phi))(f)
	\end{equation}
	for \(f \in \Gamma (U, \mc{O}_X)\).
	Next, consider the isomorphism of groups (see \cite[Lemma 3.4]{MO}) 
	\begin{equation*}
		a_X	:\text{Lie}(\text{Aut}(X)) \rar 	H^0(X, \mathscr{T}_X) \text{ given by } \phi \mapsto \delta_{\phi}, 
	\end{equation*}
where \(\delta_{\phi}\) is determined by the rule 
\begin{equation*}\label{1.1}
	\phi(f)= f + \varepsilon \, \delta_{\phi}(f) \in \Gamma (U, \mc{O}_X + \varepsilon \, \mc{O}_X)= \Gamma(U, \text{I}_{\mathbb{K}} \otimes \mc{O}_X)
\end{equation*}
	for \(f \in \Gamma(U, \mc{O}_X)\), \(U\) open subset of \(X\). Similarly, we have an isomorphism 
	\begin{equation*}
		a_{\mathcal{P}}	:\text{Lie}(\text{Aut}^G(\mathcal{P})) \rar 	H^0(X, \mathcal{A}t(\mathcal{P})) \text{ given by } \phi \mapsto \delta_{\phi} 
	\end{equation*}
	where \(\delta_{\phi}\) is determined by the rule 
	\begin{equation}\label{con1.1}
		\phi(f)= f + \varepsilon \, \delta_{\phi}(f) \in \Gamma (U, \mc{O_P} + \varepsilon \, \mc{O_P})= \Gamma(U, \text{I}_{\mathbb{K}} \otimes \mc{O_P})
	\end{equation}
	for \(f \in \Gamma(U, \mc{O_P})\), \(U\) open subset of \(\mathcal{P}\).

	Recall the following proposition on the existence of an equivariant structure on a principal bundle whose proof follows verbatim from the proof of \cite[Proposition 4.1]{BDP}.
	\begin{prop}\cite[Proposition 4.1]{BDP}\label{condn_equ}
		Let \(G\) be reductive. A principal \(G\)-bundle $\mathcal{P}$ on a projective toric variety \(X\) can be endowed with an equivariant structure if 
		\begin{equation*}
			\mathcal{P} \, \cong \, \varphi_t^* \mathcal{P}, \text{ for all } t \in T,
		\end{equation*}
		where \(\varphi_t: X \rightarrow X\) is given by sending \(x \in X \) to \(tx\). In other words, \(\mathcal{P}\) is \(T\)-equivariant if 
		\begin{equation*}
			T \subseteq \text{Im}(h).
		\end{equation*}
		
	\end{prop}
	
	\begin{thm}\label{logconn2equi}
		Let \(X\) be a projective toric variety. Let \(\mathcal{P}\) be a $T$-equivariant algebraic principal $G$-bundle over $X$, where \(G\) is reductive. Then the following statements are equivalent:
		\begin{enumerate}[$(i)$]
			\item The principal \(G\)-bundle $\mathcal{P}$ admits an equivariant structure.
			
			\item The principal \(G\)-bundle $\mathcal{P}$ admits an integrable logarithmic connection singular over \(D\).
			
			\item The principal \(G\)-bundle $\mathcal{P}$ admits a logarithmic connection singular over \(D\).
		\end{enumerate}
	\end{thm}
	
	\begin{proof}
		Using Proposition \ref{prop1}, \((i)\) implies \((ii)\). Clearly, \((ii)\) implies \((iii)\). We now show that \((iii)\) implies \((i)\).
		
		Let 
		\begin{equation*}
			\zeta : \mathscr{T}_X(- \text{log } D) \rightarrow \mathcal{A}t(\mathcal{P}) (- \text{log } D) 
		\end{equation*}
		be a logarithmic connection singular over \(D\) i.e., \(\bar{\eta} \circ \zeta = Id_{\mathscr{T}_X(- \text{log } D)}\).

		Note that we have the following commutative diagram:
		
		\begin{center}
			\begin{figure}[h] 
				\begin{tikzpicture}
					\matrix (m) [matrix of math nodes,row sep=3em,column sep=4em,minimum width=2em] {
						\text{Lie}(\text{Aut}^G(\mathcal{P})) & 	\text{Lie}(\text{Aut}(X)) \\
						H^0(X, \mathcal{A}t(\mathcal{P})) & H^0(X,\mathscr{T}_X) \\};
					\path[-stealth]
					(m-1-1) edge node [left] {$a_{\mathcal{P}}$} (m-2-1)
					edge  node [above] {$dh$} (m-1-2)
					(m-2-1) edge node [below] {$\eta(X)$} (m-2-2)
					(m-1-2) edge node [right] {$a_X$} (m-2-2);
				\end{tikzpicture}	
				\caption{}
				\label{scfig}
			\end{figure}
		\end{center}
		To see this consider $\phi \in \text{Lie}(\text{Aut}^G(\mathcal{P}))$. Now \(a_X(dh(\phi))=\delta_{dh(\phi)}\) is defined by the rule 
		\begin{equation*}
			dh(\phi)(f)=f + \varepsilon \, \delta_{dh(\phi)}(f)
		\end{equation*}
		for \(f \in \Gamma(U, \mc{O}_X)\), \(U\) open subset of \(X\). Hence, we have 
		\begin{equation*}
			p^{\sharp}(dh(\phi)(f))=p^{\sharp}(f) + \varepsilon \, p^{\sharp}(\delta_{dh(\phi)}(f)).
		\end{equation*}
		Thus from \eqref{key}, we have
		\begin{equation*}
			\phi(p^{\sharp}(f))=p^{\sharp}(f) + \varepsilon \,  p^{\sharp}(\delta_{dh(\phi)}(f)).
		\end{equation*}
		By \eqref{con1.1}, this shows that,
		\begin{equation}\label{der1}
			\delta_{\phi}(p^{\sharp}(f))=p^{\sharp}(\delta_{dh(\phi)}(f))=\delta_{dh(\phi)}(f) \circ p.
		\end{equation}
		Again, since \(p: \mathcal{P} \rar X\) is a good quotient, it follows that there exists \(p_*(\delta_{\phi}(p^{\sharp}(f))) \in \Gamma(U, \mc{O}_X)\) such that
		\begin{equation}\label{der2}
			\delta_{\phi}(p^{\sharp}(f))= p_*(\delta_{\phi}(p^{\sharp}(f))) \circ p.
		\end{equation}
		Since \(p\) is surjective, from \eqref{der1} and \eqref{der2}, we have
		\begin{equation}\label{dereq}
			\delta_{dh(\phi)}(f)=p_*(\delta_{\phi}(p^{\sharp}(f))).
		\end{equation}
		
		Then for $f \in 	\Gamma(U, \mathcal{O}_X),  \, U \subset X$, from \eqref{dereq}, we have
		\begin{equation*}
			\begin{split}
				\eta(X)(a_{\mathcal{P}}(\phi))(f)=\eta(X)(\delta_{\phi}) (f)&= p_*(\delta_{\phi}(p^{\sharp}(f))) \\
				&=\delta_{dh(\phi)}(f)\\
				&=a_X(dh(\phi))(f).
			\end{split}
		\end{equation*}

		\noindent
		This shows that the diagram in Figure \ref{scfig} commutes.
		
		Thus we have an injection 
		\begin{equation*}
			\theta : \mathfrak{t}=H^0(X,\mathscr{T}_X(- \text{log }D)) \rar \text{Lie}(\text{Aut}^G(\mathcal{P})), 
		\end{equation*} given by the composition
		\begin{equation*}
			H^0(X,\mathscr{T}_X(- \text{log }D)) \stackrel{\zeta(X)} \longrightarrow H^0(X, \mathcal{A}t(\mathcal{P}) (- \text{log } D) ) \hookrightarrow H^0(X, \mathcal{A}t(\mathcal{P}))  \stackrel{{a_{\mathcal{P}}^{-1}}} \longrightarrow \text{Lie}(\text{Aut}^G(\mathcal{P})).
		\end{equation*}
		From the commutative diagram Figure \ref{scfig}, it follows that 
		\begin{equation*}
			a_X \circ dh \circ \theta =\text{id}_{\mathfrak{t}}.
		\end{equation*}
		Since $\theta$ is an injection and \(a_X\) is an isomorphism, we have \(\mathfrak{t} \subset \text{Im}(dh).\) This implies that 
		\begin{equation*}
			T \subseteq \text{Im}(h),
		\end{equation*}
		as \(T\) is connected (see \cite[Proposition 24.5.3]{TY}). Hence, by Proposition \ref{condn_equ}, $\mathcal{P}$ admits an equivariant structure.
		
	\end{proof}

	\subsection{Residue of logarithmic connections on vector bundles over a toric variety}\label{ssc: residue}
	In this section, we give an explicit description of residues of logarithmic connection on a vector bundle over a projective toric variety.
	Let \(\pi : E \rightarrow X\) be a vector bundle of rank \(r\) on a toric variety \(X\). Let 
	\begin{equation}\label{conn}
		\nabla : E \longrightarrow E \otimes_{\mathcal{O}_X} \Omega^1_X(\text{log } D)
	\end{equation}
	be a \(\mathbb{K}\)-linear sheaf homomorphism satisfying the Leibniz rule 
	\begin{equation*}
		\nabla(f \, s )=f \, \nabla(s) + s \otimes df,
	\end{equation*}
	for any open subset \(U\) of \(X\) and all sections \(f \in \mathcal{O}_X(U) \) and \(s \in \Gamma(U, E)\). For each $\rho \in \Delta(1)$, we have the natural inclusion,
	\begin{equation*}
		i_{\rho}:D_{\rho} \rar X
	\end{equation*}
	which gives the natural surjective map 
	\begin{equation}\label{res}
		\mathcal{O}_X \rar \mathcal{O}_{D_{\rho}}.
	\end{equation}
	We have the natural map 
	\begin{equation*}
		e_{\rho}:	M \ \rar \ \Z \ \ \text{given by } m \mapsto \langle m, v_{\rho} \rangle.
	\end{equation*} Thus we have a map of \(\mathcal{O}_X\)-modules given by 
	\begin{equation*}
		\Omega_X^1(\text{log } D)\cong	M \otimes_{\Z} \mathcal{O}_X \ \stackrel{e_{\rho} \otimes id}{\longrightarrow} \ \Z \otimes_{\Z} \mathcal{O}_X \cong \mathcal{O}_X.
	\end{equation*}
	Composing this with \eqref{res}, we get a map of \(\mathcal{O}_X\)-modules
	\begin{equation*}
		\Omega_X^1(\text{log } D) \rar \mathcal{O}_{D_{\rho}}.
	\end{equation*}
	Tensoring this with the locally free sheaf \(E\), we have a map
	\begin{equation*}
		E \otimes_{\mathcal{O}_X}\Omega_X^1(\text{log } D) \rar E \otimes_{\mathcal{O}_X} \mathcal{O}_{D_{\rho}}.
	\end{equation*}

Composing  with \eqref{conn}, we have the following map of \(\mathcal{O}_X\)-modules
\begin{equation*}
E \stackrel{\nabla} \rightarrow	E \otimes_{\mathcal{O}_X}\Omega_X^1(\text{log } D) \rar E \otimes_{\mathcal{O}_X} \mathcal{O}_{D_{\rho}}.
\end{equation*}

Finally restricting to \(D_{\rho}\), we have the following map of \(\mathcal{O}_{D_{\rho}}\)-modules

	\begin{equation}\label{defn_res} 
		E|_{D{\rho}}=	E \otimes_{\mathcal{O}_X} \mathcal{O}_{D_{\rho}} \, {\longrightarrow} \, E \otimes_{\mathcal{O}_X} \mathcal{O}_{D_{\rho}}= E|_{D{\rho}}.
	\end{equation} 
	Thus we define the residue of \(\nabla\) along \(D_{\rho}\) as the section \(\Gamma_{D{\rho}} \in H^0(D_{\rho}, End(E|_{D{\rho}}))\) coming from \eqref{defn_res}.

	
	By Theorem \ref{logconn2equi} and Corollary \ref{equivalence}/ Proposition \eqref{equivalence1}, we have an equivariant structure on the vector bundle \(E\). Recall that, for each \(\sigma \in \Sigma\), we have the distinguished point \(x_{\sigma}\) in the affine toric variety \(U_{\sigma}\). Let \(E(x_{\sigma})\) be the fiber of the vector bundle \(E\) over the point \(x_{\sigma}\). Since the vector bundle \(E\) is \(T\)-equivariant, we have an action of \(T_{\sigma}\) on \(E(x_{\sigma})\). Now consider the evaluation map,
	\begin{equation*}
		ev_{x_{\sigma}}: \Gamma(U_{\sigma}, E) \, \rar \, 	E(x_{\sigma}),
	\end{equation*}
	given by evaluating the sections at the distinguished point \(x_{\sigma}\). There is a natural action of \(T\) on the space of sections \(\Gamma(U_{\sigma}, E)\), 
	\begin{equation*}
		\begin{split}
			T \times \Gamma(U_{\sigma}, E) \ &\rar \ \Gamma(U_{\sigma}, E)\\
			(t, s) & \mapsto \ t \cdot s
		\end{split}
	\end{equation*} defined by 
	\begin{equation*}
		(t \cdot s)(x)=t \cdot s(t^{-1} \cdot x)
	\end{equation*}
	for all \(t \in T, \, s \in \Gamma(U_{\sigma}, E), \, x \in U_{\sigma}\). Following the proof of \cite[Proposition 2.1.1]{kly}, let $V$ be a maximal $T$-stable subspace of \(\Gamma(U_{\sigma}, E)\) on which \(ev_{x_{\sigma}}\) is an isomorphism (existence of non-empty $V$ follows from the complete reducibility of the torus).  Again, by complete reducibility, we have an isotypical decomposition \begin{equation*}
		V=\oplus_{m \in M} V_m.
	\end{equation*}
	Thus we can choose a basis \(s_1, \, \ldots \, s_r\) of \(V\) consisting of eigenvectors so that \(ev_{x_{\sigma}}(s_1), \, \ldots \, ev_{x_{\sigma}}(s_r)\) forms a basis of \(E(x_{\sigma})\). Suppose that the weight of the section \(s_i\) is \(m_i\) for all \(i=1, \, \ldots, \, r\), i.e.,
	\begin{equation*}
		(t \cdot s_i)(x)= \chi^{m_i}(t) s_i( x)
	\end{equation*}
	for all \(t \in T\), \(x \in U_{\sigma}\) and some \(m_i \in M\). Following \cite{BMS}, define new sections \(\tilde{s}_1, \, \ldots, \, \tilde{s}_r\) over the open orbit \(O\) as
	\begin{equation}\label{invariant}
		\tilde{s}_i(t \cdot x_0)=\chi^{m_i}(t)s_i(t \cdot x_0)
	\end{equation}
	for \(t \in T\). Then we see that 
	\begin{equation*}
		\tilde{s}_i( x_0)=s_i(x_0)
	\end{equation*}
	and moreover $\tilde{s}_i$ is a \(T\)-invariant section over the open orbit. To see this, let \(t \in T\) and \(x \in  O\), then writing \(x=t' \cdot x_0\) we have
	\begin{equation*}
		\begin{split}
			(t \, \tilde{s}_i)(t' \cdot x_0)= & t \,  \tilde{s}_i(t^{-1}t'x_0)=t \,  \chi^{m_i}(t^{-1}t') \, s_i(t^{-1}t'x_0)\\
			= & \chi^{m_i}(t^{-1}t') \, t \, s_i(t^{-1} x)=\chi^{m_i}(t^{-1}t') \, (t \cdot s_i)(x)\\
			= & \chi^{m_i}(t^{-1}t') \, \chi^{m_i}(t) \, s_{i}(x)= \chi^{m_i}(t') \, s_{i}(t' x_0)\\
			= & \tilde{s}_i(t' x_0)=\tilde{s}_i(x).
		\end{split}
	\end{equation*}
	
	Hence, 
	\(\nabla(\tilde{s}_i)=0\) (see Remark \ref{vanishing}). By \eqref{invariant} and the Leibniz rule, we have
	\begin{equation*}
		\nabla(s_i)=-s_i \otimes \frac{\text{d} \chi^{m_i}}{\chi^{m_i}}
	\end{equation*}
	over the open orbit. By continuity, 
	\begin{equation*}
		\nabla(s_i)=-s_i \otimes \frac{\text{d} \chi^{m_i}}{\chi^{m_i}}.
	\end{equation*}
	on \(U_{\sigma}\). Note that, we have the identification
	\begin{equation*}
		\Omega_X^1(\text{log } D)(U_{\sigma}) \, \rar \, M \otimes \mathbb{K}[S_{\sigma}] \text{ given by } \frac{\text{d}\chi^{m_i}}{\chi^{m_i}} \mapsto m_i \otimes 1.
	\end{equation*} Fix \(\rho \in \sigma(1)\). Restricting to \(D_{\rho}\), we have 
	\begin{equation*}
		\nabla|_{D_{\rho}}(s_i)=-s_i \otimes ( m_i \otimes 1).
	\end{equation*} Then by definition of residue \eqref{defn_res}, over \(U_{\sigma} \cap D_{\rho}\) we obtain
	\begin{equation*}
		\Gamma_{D_{\rho}}(s_i)=- \langle m_i, \, v_{\rho} \rangle \, s_i.
	\end{equation*}
	Thus, when \(\sigma\) is a full dimensional cone in \(N_{\R}\), we can recover the collection \(\{m_i\}\). Hence, using \cite[Corollary 2.3]{payne}, the action of \(T\) on \(E|_{U_{\sigma}}\) can be recovered from the collection of residues \(\{	\Gamma_{D_{\rho}} ~:~ \rho \in \sigma(1) \}\) . Also, note that the action of \(T\) determines the equivariant Chern class of the bundle \(E\) using \cite[Proposition 3.1]{payne}.
	
	Now, we describe how to relate the residues with the equivariant Chern class. First note that the matrix representation of \(\Gamma_{D_{\rho}}|_{U_{\sigma}\cap D_{\rho}}\) is the diagonal matrix \(\text{diag}(-\langle m_1, v_{\rho} \rangle, \ldots, -\langle m_r, v_{\rho} \rangle)\) with respect to the ordered basis \(s_1, \ldots, s_r\) of \(V\). Let \(f_{\rho}\) denote the characteristic polynomial of \(-\Gamma_{D_{\rho}}|_{U_{\sigma}\cap D_{\rho}}\). Then the coefficients of the polynomial \((-1)^r f_{\rho}\) are elementary symmetric polynomials in \(\langle m_i, v_{\rho} \rangle \) for \(i=1, \ldots, r \). 
	
	On the other hand, associated with the equivariant structure on the vector bundle, there is a collection of a multiset of linear functions, indexed by the cones in the fan, defined as follows (see \cite[Section 2.2]{payne}). From the isotypical decomposition
	\[V= \oplus \, V^{\sigma}_{m_i},\]
	we have the multiset \({\bf u}(\sigma)=\{m_1, \ldots, m_r\}\). Write \(\Psi_E=\{{\bf u}(\sigma)\}_{\sigma \in \Sigma}\) for the collection of multisets for all cones \(\sigma\) in the fan \(\Sigma\). Let \[c_i(\Psi_{E}): | \Sigma| \rar \R\] be the piecewise polynomial function whose restriction to \(\sigma\) is the \(i\)-th
	elementary symmetric function in the multiset of linear functions \({\bf u}(\sigma)\). Define
	\[c(\Psi_E)= 1 + c_1(\Psi_E) + \cdots + c_r(\Psi_E).\]
	Then by \cite[Proposition 3.1]{payne}, the equivariant total Chern class of \(E\) can be identified with \(c(\Psi_E)\). Observe that \[c_i(\Psi_{E})(v_{\rho})= i \text{-th coefficient of }(-1)^r f_{\rho}.\]
	
	This exhibits a relation between the residues and the equivariant total Chern class of the vector bundle.

	
\end{document}